\documentclass[11pt,reqno]{amsart}
\usepackage{amssymb,latexsym,amsmath,epsfig,amsthm,mathrsfs}
\usepackage{rotating}
\usepackage{graphicx}
\usepackage{amssymb}
\usepackage{lineno}
\usepackage{enumitem}
\usepackage{cite}
\usepackage{multicol}
\usepackage[top=3cm, bottom=3cm, left=2.5cm, right=2.5cm]{geometry}
\usepackage[usenames]{color}
\usepackage[colorlinks=true,
linkcolor=blue,
filecolor=blue,
citecolor=blue]{hyperref}

\numberwithin{equation}{section}

\newtheorem{theorem}{Theorem}[section]
\newtheorem{lemma}[theorem]{Lemma}

\newtheorem{conjecture}[theorem]{Conjecture}

\begin{document}
	
	\title[Freiman's $(3k-4)$-like results for subset and subsequence sums]{Freiman's $(3k-4)$-like results for subset and subsequence sums}
	
	\author{Mohan}
	\address{Department of Mathematics, Indian Institute of Technology Roorkee, Uttarakhand, 247667, India}
	\email{mohan98math@gmail.com}
	
	\author{Jagannath Bhanja}
	\address{Department of Mathematics, Indian Institute of Information Technology, Design and Manufacturing, Kancheepuram, Chennai-600127, India}
	\email{jagannath@iiitdm.ac.in}
	
	\author{Ram Krishna Pandey}
	\address{Department of Mathematics, Indian Institute of Technology Roorkee, Uttarakhand, 247667, India}
	\email{ram.pandey@ma.iitr.ac.in}

	\subjclass[2020]{11P70, 11B75, 11B13}
	
	
	
	\keywords{Subset sum; subsequence sum; sumset; restricted sumset; Freiman's $3k-4$ theorem}
	
	\begin{abstract}
	For a nonempty finite set $A$ of integers, let $S(A) = \left\{ \sum_{b\in B} b: \emptyset \not= B\subseteq A\right\}$ be the set of all nonempty subset sums of $A$. In 1995, Nathanson determined the minimum cardinality of $S(A)$ in terms of $|A|$ and described the structure of $A$ for which $|S(A)|$ is the minimum. He asked to characterize the underlying set $A$  if $|S(A)|$ is a small increment to its minimum size. Problems of such nature are inspired by the well-known Freiman's $3k-4$ theorem. In this paper, some results in the direction of Freiman's $3k-4$ theorem for the set of subset sums $S(A)$ are proved. Such results are also extended to the set of subsequence sums $S(\mathbb{A}) = \left\{ \sum_{b\in \mathbb{B}} b: \emptyset \not= \mathbb{B} \subseteq \mathbb{A} \right\}$ of sequence $\mathbb{A}$, where the notation $\mathbb{B} \subseteq \mathbb{A} $, is used for $\mathbb{B}$ is a subsequence of $\mathbb{A}$. The results are further generalized to a generalization of subset and subsequence sums. The main idea of the proofs of the results is to write the set of subset sums $S(A)$ and the set of subsequence sums $S(\mathbb{A})$ in terms of the $h$-fold sumset $hA$ and the $h$-fold restricted sumset $h^\wedge A$. Such representation also gives other proof of some of the results of Nathanson and Mistri {\it et al}. 
\end{abstract}
\maketitle

\section{Notation}
Throughout the paper, we follow the following notations. We write $\mathbb{N}$ for the set of natural numbers and $\theta$ for the golden mean $\frac{1+\sqrt{5}}{2}$. For a nonempty finite set $A$ of integers, by $|A|$ we mean the number of elements in $A$. For integers $\alpha$ and  $\beta$, we let $\alpha \ast A = \lbrace \alpha a : a\in A \rbrace$, $ A + \beta = \lbrace a + \beta  : a \in A \rbrace$, and $[\alpha, \beta] = \lbrace \alpha, \alpha +1, \ldots, \beta \rbrace$ for $\alpha \leq \beta$. For a sum of the form $\sum_{x=a}^{b} f(x)$ with integers $a, b$ such that $a > b$, we mean zero. We denote the greatest common divisor of the integers $x_{1}, x_{2}, \ldots, x_{k}$ by $(x_{1},x _{2}, \ldots, x_{k})$ and write $d(A)$ in short, when $A=\{x_{1}, x_{2}, \ldots, x_{k}\}$ is a set. For a given set $A = \{a_0, a_1, \ldots, a_{k-1}\}$ of integers with $a_0 < a_1 < \cdots < a_{k-1}$, we define $A^{(N)} := \left\{a^{\prime}_{i}= \dfrac{a_{i}-a_{0}}{d(A-a_{0})}: a_{i} \in A \right\}$. Then $|A^{(N)}|=k$, $0 \in A^{(N)}$, and $d(A^{(N)})=1$. We call the set $A^{(N)}$ the normal form of $A$. Further, we write $\mathbb{A}=\{a_{1}, a_{2}, \ldots, a_{k}\}_{\vec{r}}$ with $a_{1} < a_{2} < \cdots < a_{k}$ and $\vec{r}=(r_{1}, r_{2}, \ldots, r_{k})$ to mean that $\mathbb{A}$ is a sequence consisting of $k$ distinct integers $a_{1}, a_{2}, \ldots, a_{k}$ with  $a_{i}$ appearing $r_{i}$ times in $\mathbb{A}$ for $i=1,2,\ldots,k$. We use the usual set notation $A$ to write the set $\{a_{1}, a_{2}, \ldots, a_{k}\}$ of distinct elements of sequence $\mathbb{A}$. By $|\mathbb{A}|$ we mean the number of terms (including multiplicities) in $\mathbb{A}$. By the size of sequence $\mathbb{A}=\{a_{1}, a_{2}, \ldots, a_{k}\}_{\vec{r}}$ we mean the number $\sum_{i=1}^{k} r_{i}$.

\section{Introduction}
Let $A$ be a nonempty finite set of integers and $h$ be a positive integer. The \textit{h-fold sumset}, denoted by $hA$, is defined as the set of integers that can be written as a sum of $h$ elements (not necessarily distinct) of $A$, and the \textit{restricted h-fold sumset}, denoted by $h^{\wedge}A$, is defined as the set of integers that can be written as a sum of $h$ distinct elements of $A$ (see \cite{Bajnok2018, Nathanson1996}).

Two problems associated with sumsets that are studied extensively in the literature are direct and inverse problems. A direct problem is to determine the minimum cardinality and properties of the sumset and the inverse problem is to characterize the underlying set(s) when the cardinality of the sumset is known. The following are some of the classical results that give the minimum cardinality of $h$-fold sumset $hA$ and $h^\wedge A$ and also describe the underlying set $A$ when the cardinality of the sumset is minimum.

\begin{theorem}\textup{\cite[Theorem 1.4, Theorem 1.6]{Nathanson1996}}\label{Direct and Inverse Thm Nathanson}
	Let $A$ be a nonempty finite set of integers. Then, for $h\geq 1$, we have 
	\begin{center}
		$\left| hA\right| \geq h\left| A\right| - h + 1$.
	\end{center}
	Moreover, if $h \geq 2$ and $|hA|=h|A|-h+1$, then $A$ is an arithmetic progression.
\end{theorem}

\begin{theorem} \textup{\cite[Theorem 1.9, Theorem 1.10]{Nathanson1996}} \label{Direct and Inverse Thm Nathanson restricted}
	Let A be a nonempty finite set of integers, and $1 \leq h \leq \left| A\right| $. Then 
	\begin{center}
		$\left| h^{\wedge}A \right| \geq h \left| A\right| - h^{2} + 1$.
	\end{center}
	Moreover, if $\left| h^{\wedge}A\right| = h \left| A\right| - h^{2} + 1$ with $\left| A\right| \geq 5$ and $2 \leq h \leq \left| A\right| -2$, then A is an arithmetic progression.
\end{theorem}

Freiman \cite{GAF1959, GAF1973} proved the following inverse theorem for the $2$-fold sumset $2A$, which is well known as Freiman's $3k-4$ theorem. 
\begin{theorem}\textup{\cite[Theorem 1.9]{GAF1973}}\label{Theorem Freiman II}
	Let $k \geq 3$. Let $A$ be a set of $k$ integers. If  $\left|2A\right|=2k-1+b \leq 3k-4$, then $A$ is a subset of an arithmetic progression of length at most  $k+b$.
\end{theorem}

This inverse theorem is a consequence of the following result. 

\begin{theorem}\textup{\cite[Theorem 1.10]{GAF1973}}\label{Theorem Freiman I}
	Let $k \geq 3$. Let $A=\lbrace a_{0}, a_{1}, \ldots, a_{k-1}\rbrace$ be a set of integers such that $0=a_{0}<a_{1}< \cdots < a_{k-1}$ and $d(A)=1$. Then
	\[ |2A| \geq 
	\begin{cases}
		a_{k-1}+k, & \text{ if } a_{k-1} \leq 2k-3 \\
		3k-3 & \text{ if }  a_{k-1} \geq 2k-2.
	\end{cases}\]
\end{theorem}	

Lev \cite{Lev1996} extended Theorem \ref{Theorem Freiman I} to the sumsets $hA$ for $h \geq 2$. 

\begin{theorem}\textup{\cite[Theorem 1]{Lev1996}}\label{Theorem Lev 1996}
	Let $k \geq 3$. Let $A=\lbrace a_{0}, a_{1}, \ldots, a_{k-1}\rbrace$ be a set of integers such that $0=a_{0}<a_{1}< \cdots < a_{k-1}$ and $d(A)=1$. Then, for $h \geq 2$, we have
	\[ \left|hA\right| \geq |(h-1)A| + \min \{ a_{k-1}, h(k-2)+1\}. \]
\end{theorem}	

For the restricted sumset $2^{\wedge}A$, the following was conjectured by Freiman and Lev \cite{Lev2000}, independently. 

\begin{conjecture}\label{Conjecture 1}
	Let $k > 7$. Let $A=\{a_{0}, a_{1}, \ldots, a_{k-1}\}$ be a set of integers such that $0=a_{0}<a_{1}< \cdots < a_{k-1}$ and $d(A)=1$. Then
	\[
	\left|2^{\wedge}A\right| \geq 
	\begin{cases}
		a_{k-1} + k-2, & \text{ if } a_{k-1} \leq 2k-5 \\
		3k-7   & \text{ if }  a_{k-1} \geq 2k-4.
	\end{cases}
	\]
\end{conjecture}

The lower bounds in Conjecture \ref{Conjecture 1} are tight, as letting $A = \{0,1, \ldots, k-3\} \cup \{a_{k-1}-1, a_{k-1}\}$, we get $2^{\wedge}A = \{1,2, \ldots, 2k-7\} \cup \{a_{k-1}-1, \ldots, a_{k-1} + k-3\} \cup \{2a_{k-1}-1\}$. Freiman {\it et al}. \cite{FLP1999} made some progress on Conjecture \ref{Conjecture 1} by proving the following result. 

\begin{theorem}\textup{\cite[Theorem 1, Theorem 2]{FLP1999}}\label{FLP restricted thm}
	Let $k \geq 3$. Let $A=\lbrace a_{0}, a_{1}, \ldots, a_{k-1}\rbrace$ be a set of integers such that
	$0=a_{0}<a_{1}< \cdots < a_{k-1}$ and $d(A)=1$. Then
	\[
	\left|2^{\wedge}A\right| \geq 
	\begin{cases}
		0.5(a_{k-1}+k)+k-3.5 & \text{ if } a_{k-1} \leq 2k-3 \\
		2.5 k-5   & \text{ if }  a_{k-1} \geq 2k-2.
	\end{cases}
	\]
\end{theorem}

A year later, Lev \cite{Lev2000} improved Freiman {\it et al}. \cite{FLP1999} results in the following theorem.

\begin{theorem}\textup{\cite[Theorem 1]{Lev2000}}\label{Lev restricted thm}
	Let $k \geq 3$. Let $A=\lbrace a_{0}, a_{1}, \ldots, a_{k-1}\rbrace$ be a set of integers such that
	$0=a_{0}<a_{1}< \cdots < a_{k-1}$ and $d(A)=1$. Then
	\[
	\left|2^{\wedge}A\right| \geq  
	\begin{cases}
		a_{k-1} + k-2 & \text{ if } a_{k-1} \leq 2k-5 \\
		(\theta +1)k-6   & \text{ if }  a_{k-1} \geq 2k-4.
	\end{cases}
	\]
\end{theorem}

The purpose of this article is to prove results similar to Theorem \ref{Theorem Freiman I} and Theorem \ref{Lev restricted thm} for the set of subset sums and the set of subsequence sums, which are defined below. 

Let $A$ be a nonempty finite set of $k$ integers. For a nonempty subset $B$ of $A$, the \textit{subset sum} of $B$ is defined as $s(B) = \sum_{b \in B} b$. The collection of all nonempty subset sums of $A$, denoted by $S(A)$, is defined as 
\[S(A) := \Big\{s(B) : \emptyset \not= B\subseteq A\Big\}. \] 

Nathanson \cite{Nathanson1995} initiated the study of direct and inverse problems for $S(A)$ over the group of integers. Such studies are done on other groups also (see \cite{Balandraud2012, Balandraud2013, Devos2007, Griffiths2009}, and the references therein). However, in this article, we restrict ourselves to the group of integers only. Nathanson \cite{Nathanson1995} determined the minimum cardinality of $S(A)$ in terms of $|A|$, and also gave a characterization of set $A$ when $|S(A)|$ is the minimum (see Lemma \ref{Lemma A}). Lev \cite{Lev1999} extended Nathanson's direct theorem to sequences of nonnegative integers (see Lemma \ref{Direct-lemma-for-vector-r}). Mistri {\it et al}. \cite{MistriPandey2015} (also see \cite{MistriPandey2016}) extended Nathanson's inverse theorem to sequences of nonnegative integers (see Lemma \ref{Inverse-lemma-for-vector-r} and Lemma \ref{Direct-lemma-for-vector-r}) while giving a new proof of Lev's result. Jiang and Li \cite{JiangLi2018} later proved the direct and inverse results for the subsequence sums when the sequence contains both positive and negative integers. For the sake of  completeness, we define the subsequence sums below. 

For a nonempty finite sequence $\mathbb{A}$ of integers, we denote by $S(\mathbb{A})$,  the set of all subsequence sums of $\mathbb{A}$, i.e., 
\[ S(\mathbb{A}) := \{s(\mathbb{B}): \emptyset \neq \mathbb{B} \subseteq \mathbb{A}\}, \] where $s(\mathbb{B}) = \sum_{b \in \mathbb{B}} b$. 

The direct and inverse results for the usual subset and subsequence sums are further extended by Bhanja and Pandey \cite{BhanjaPandey2020, BhanjaPandey2022} considering the $\alpha$-analog of subset and subsequence sums, which are defined below. For a given positive integer $\alpha$, let
\[ S_{\alpha}(A) := \{s(B): B \subseteq A, |B| \geq \alpha\}, \]
and
\[ S_{\alpha}(\mathbb{A}) := \{s(\mathbb{B}): \mathbb{B} \subseteq \mathbb{A}, |\mathbb{B}| \geq \alpha\}. \] 
Recently, Dwivedi and Mistri \cite{Dwivedi2022} reproved some results of Bhanja and Pandey using a generalization of $h$-fold sumset $hA$. The reader is also directed to see the article of Balandraud \cite{Balandraud2017}, where $S_{\alpha}(A)$ is introduced in this context, and also the minimum cardinality of $S_{\alpha}(A)$ is obtained over the finite cyclic groups of prime order. 

In \cite{Nathanson1995}, Nathanson asked to prove Theorem \ref{Theorem Freiman I}-like result for $S(A)$. In this paper, we prove some results (see Theorems \ref{Theorem-S(A)-1}, \ref{Theorem-S(A)-2}, \ref{Freiman-theorem-1-for-vector-r}, \ref{Freiman-theorem-2-for-vector-r}, and \ref{Freiman-theorem-for-S(A)-zero-1}) for $S(A)$ and $S(\mathbb{A})$ which are similar to Theorem \ref{Theorem Freiman I} and Theorem \ref{Lev restricted thm}. Our idea is to write $S(A)$ and $S(\mathbb{A})$ in terms of sumsets $hA$ and $h^\wedge A$, and then use Theorem \ref{Theorem Freiman I} and Theorem \ref{Lev restricted thm} to obtain Freiman like results for $S(A)$ and $S(\mathbb{A})$. Such representation will also lead us to give new proofs of some results of Nathanson \cite{Nathanson1995} (see Lemma \ref{Lemma A}) and Mistri et al. \cite{MistriPandey2015} (see Lemma \ref{Inverse-lemma-for-vector-r} and Lemma \ref{Direct-lemma-for-vector-r}). Further, we prove analogous results for $S_{\alpha}(A)$ and $S_{\alpha} (\mathbb{A})$ in the last two sections of this paper. The proofs of the results of sections \ref{section-alpha-subset-sum} and \ref{section-alpha-subsequence-sum} are quite similar to the ones in sections \ref{section-subset-sum} and \ref{section-subsequence-sum}, however, in sections \ref{section-alpha-subset-sum} and \ref{section-alpha-subsequence-sum} the proofs are more involved and depend heavily on $\alpha$. 

To prove the main results of sections \ref{section-subset-sum} and \ref{section-subsequence-sum} of this article, we first reprove the direct and inverse results for the usual subset and subsequence sums. In other results that we prove in sections \ref{section-alpha-subset-sum} and \ref{section-alpha-subsequence-sum}, we directly use the already proven results for the $\alpha$-analog of subset and subsequence sums. The following are the two results that we use to prove our results in sections \ref{section-alpha-subset-sum} and \ref{section-alpha-subsequence-sum}.

\begin{theorem}\label{theorem-BP_for-subset-sums}\textup{\cite[Theorem 2.1, Theorem 2.2]{BhanjaPandey2020}}
	Let $A$ be a set of $k$ positive integers. Let $1 \leq \alpha \leq k$ be an integer. Then 
	\begin{equation*}
		\left|S_{\alpha}(A)\right| \geq \frac{k(k+1)}{2}-\frac{\alpha(\alpha+1)}{2} + 1.
	\end{equation*}
	Moreover, if $k \geq 4$, $\alpha \leq k-2$, and $\left|S_{\alpha}(A)\right| = \frac{k(k+1)}{2}-\frac{\alpha(\alpha+1)}{2} + 1$, then $A = d \ast [1,k]$ for some positive integer $d$. 
\end{theorem}

\begin{theorem}\label{theorem-BP_for-subsequence-sums}\textup{\cite[Theorem 3.1, Theorem 3.2]{BhanjaPandey2020}}
	Let $\mathbb{A} = \{a_{1},a_{2},\ldots,a_{k}\}_{\vec{r}}$ be a sequence of positive integers such that $ a_{1} < a_{2} < \cdots < a_{k}$ and $\vec{r} = (r_{1}, r_{2}, \ldots, r_{k})$ with $r_{i} \geq 1$ for all $i\in [1,k]$. Let $1 \leq \alpha \leq \sum_{i=1}^{k} r_{i}$ be an integer. Then there exists an integer $m \in [1,k]$ such that $\sum_{i=1}^{m-1} r_{i}\leq \alpha < \sum_{i=1}^{m}r_{i}$ and 
	\begin{equation*}
		|S_{\alpha}(\mathbb{A})| \geq  \sum_{i=1}^{k} ir_{i} - \sum_{i=1}^{m} ir_{i} + m\left(\sum_{i=1}^{m} r_{i}-\alpha\right) + 1.
	\end{equation*}
	Moreover, if $k \geq 4$, $\alpha \leq \sum_{i=1}^{k} r_{i}-2$, and $|S_{\alpha}(\mathbb{A})| = \sum_{i=1}^{k} ir_{i} - \sum_{i=1}^{m} ir_{i} + m(\sum_{i=1}^{m} r_{i}-\alpha) + 1$, then $\mathbb{A} = d \ast [1,k]_{\vec{r}}$ for some positive integer $d$. 
\end{theorem}

\section{Freiman's theorem for subset sum}\label{section-subset-sum}

In the following lemma, we reprove the direct and inverse results of Nathanson for $S(A)$ when the set $A$ contains positive integers. Then, in the next two theorems, we prove Freiman-like results for $S(A)$ in the cases, $A$ contains positive integers and $A$ contains nonnegative integers with $0 \in A$. 

\begin{lemma}\label{Lemma A}\textup{\cite[Theorem 3, Theorem 5]{Nathanson1995}}
	Let $A$ be a set of $k$ positive integers. Then 
	\begin{equation}\label{Bound A}
		\left|S(A)\right| \geq \frac{k(k+1)}{2}.
	\end{equation}
	Moreover, if $k \geq 4$ and $\left|S(A)\right| = \dfrac{k(k+1)}{2}$, then $A = d \ast [1,k]$ for some positive integer $d$. 
\end{lemma}

\begin{proof}
	Let $A = \{a_{1}, a_{2}, \ldots, a_{k}\}$ with $0<a_{1} < a_{2} < \cdots < a_{k}$. It is easy to see that the result holds for $k=1,2$. Assume that $k \geq 3$ and the result holds for all sets that have less than $k$ elements. Let $B = A \setminus \{a_{k-1}, a_{k}\}$. Then $2^{\wedge}(A\cup \{0\})$ and  $S(B) + a_{k-1} + a_{k}$ are two disjoint subsets of $S(A)$. By Theorem \ref{Direct and Inverse Thm Nathanson restricted} and the induction hypothesis, we get 
	\begin{align}\label{eqn-S(A)-1}
		\left| S(A) \right| &\geq \left| 2^{\wedge}(A\cup \{0\}) \right|
		+\left| S(B) + a_{k-1} + a_{k} \right| \\ \nonumber
		&\geq 2(k+1)-3 + \dfrac{(k-2)(k-1)}{2} \\ \nonumber
		&= \dfrac{k(k+1)}{2}.
	\end{align}
	
	Now, suppose that $k \geq 4$ and $\left|S(A)\right| = \dfrac{k(k+1)}{2}$. Then by (\ref{eqn-S(A)-1}), we have  $\left| 2^{\wedge}(A\cup \{0\}) \right| = 2(k+1) - 3$. Applying  Theorem \ref{Direct and Inverse Thm Nathanson restricted} on $A\cup \{0\}$, we get that $A \cup \{0\}$ is an arithmetic progression. Hence $A = a_{1} \ast [1,k]$. 
\end{proof}

\begin{theorem}\label{Theorem-S(A)-1}
	Let $k \geq 3$. Let $A = \{a_{1},a_{2},\ldots,a_{k}\}$ be a set of $k$ positive integers such that $ a_{1} < a_{2} < \cdots < a_{k}$ and $d(A)=1$. Then 		
	\[\left|S(A)\right| \geq 
	\begin{cases}
		a_{k} + \dfrac{k(k-1)}{2} & \text{ if } a_{k} \leq 2k-3 \vspace{0.2cm}\\
		\theta(k+1) - 4 + \dfrac{k(k-1)}{2} & \text{ if } a_{k} \geq 2k-2.
	\end{cases}  \]
\end{theorem}

\begin{proof}
	From equation (\ref{eqn-S(A)-1}), we have the following inequality 
	\begin{equation*}
		\left| S(A) \right| 
		\geq \left| 2^{\wedge}(A\cup \{0\}) \right| + \left| S(B) + a_{k-1} + a_{k} \right|,
	\end{equation*} 
	where $B = A \setminus \{a_{k-1}, a_{k}\}$. Applying Theorem \ref{Lev restricted thm} on $A\cup \{0\}$ and Lemma \ref{Lemma A} on $B$ we obtain 
	\begin{align*}
		\left| S(A) \right| 
		&\geq \left| 2^{\wedge}(A\cup \{0\}) \right| + \left| S(B) \right| \\ 
		&\geq 
		\begin{cases}
			a_{k}+k-1 + \dfrac{(k-1)(k-2)}{2} & \text{ if } a_{k} \leq 2(k+1)-5 \vspace{0.2cm}\\
			(\theta+1)(k+1)-6 + \dfrac{(k-1)(k-2)}{2} & \text{ if } a_{k} \geq 2(k+1)-4,
		\end{cases} \\ 
		&\geq 
		\begin{cases}
			a_{k} + \dfrac{k(k-1)}{2} & \text{ if } a_{k} \leq 2k-3 \vspace{0.2cm}\\
			\theta(k+1) - 4 + \dfrac{k(k-1)}{2} & \text{ if } a_{k} \geq 2k-2.
		\end{cases} 
	\end{align*}
\end{proof}


\begin{theorem}\label{Theorem-S(A)-2}
	Let $k \geq 4$ and $A = \{a_{0},a_{1},\ldots,a_{k-1}\}$ be a set of $k$ nonnegative integers such that $0 = a_{0} < a_{1} < \cdots < a_{k-1}$ and $d(A)=1$. Then  
	\[
	\left| S(A) \right|  
	\geq 
	\begin{cases}
		a_{k-1} + \dfrac{(k-1)(k-2)}{2} + 1
		& \text{ if } a_{k-1} \leq 2k-5 \vspace{0.2cm}\\ 
		\theta k - 3 + \dfrac{(k-1)(k-2)}{2}  & \text{ if } a_{k-1} \geq 2k-4.
	\end{cases} 
	\]
\end{theorem}

\begin{proof}
	Set $B = \{a_{1}, a_{2}, \ldots, a_{k-1}\}$. Then $B$ is a set of $k-1$ positive integers with $d(B)=1$. Further, we have $S(A) = S(B) \cup \{0\}$. Thus, by Theorem \ref{Theorem-S(A)-1}, it follows that  
	\begin{equation*}
		\left| S(A) \right| 
		\geq \left| S(B) \right| + 1 
		\geq 
		\begin{cases}
			a_{k-1} + \dfrac{(k-1)(k-2)}{2} + 1 & \text{ if } a_{k-1} \leq 2k-5 \vspace{0.2cm}\\ 
			\theta k - 3 + \dfrac{(k-1)(k-2)}{2}  & \text{ if } a_{k-1} \geq 2k-4.
		\end{cases} 
	\end{equation*}
\end{proof}


\section{Freiman's theorem for subsequence sum}\label{section-subsequence-sum}

In this section, we start with giving a new proof of direct and inverse results of Mistri {\it et al}. \cite{MistriPandey2015} for $S(\mathbb{A})$ in Lemma \ref{Inverse-lemma-for-vector-r}. Then, using Lemma \ref{Inverse-lemma-for-vector-r}, we prove a Freiman-like result for $S(\mathbb{A})$ in Theorem \ref{Freiman-theorem-1-for-vector-r} when the sequence $\mathbb{A}$ contains positive integers. In Theorem \ref{Freiman-theorem-2-for-vector-r}, we improve our previous bound assuming that every element of the sequence appears at least twice. To prove Theorem \ref{Freiman-theorem-2-for-vector-r} we first prove Lemma \ref{Direct-lemma-for-vector-r}. Further, in Theorem \ref{Freiman-theorem-for-S(A)-zero-1}, we prove a similar Freiman's $3k-4$-like theorem for $S(\mathbb{A})$ when the sequence $\mathbb{A}$ contains nonnegative integers with $0 \in \mathbb{A}$.

\begin{lemma}\label{Inverse-lemma-for-vector-r}\textup{\cite[Theorem 3.1, Theorem 3.2]{MistriPandey2015}}
	Let $\mathbb{A} = \{a_{1},a_{2}, \ldots, a_{k}\}_{\Vec{r}}$ be a sequence of positive integers such that $ a_{1}< a_{2}< \cdots < a_{k}$, $\vec{r} = (r_{1}, r_{2}, \ldots, r_{k})$ and $r_{i} \geq 1$ for all $i \in [1,k]$. Then 
	\begin{equation}\label{Equation-S-vec-A-1}
		\left|S(\mathbb{A}) \right| \geq \sum_{i=1}^{k}ir_{i}.
	\end{equation} 
	
	Moreover, if $k \geq 4$ and $\left|S(\mathbb{A}) \right| = \sum_{i=1}^{k}ir_{i}$, then $\mathbb{A} = d\ast [1,k]_{\Vec{r}}$ for some positive integer $d$. 
\end{lemma}

\begin{proof}	
	To prove (\ref{Equation-S-vec-A-1}), we  use induction on $k$. For $k=1$, we have $\mathbb{A} = (a_{1})_{\Vec{r_{1}}}$,  and so $S(\mathbb{A} )= \{a_{1}, 2a_{1}, \ldots, r_{1}a_{1}\} $. For $k=2$, we have $\mathbb{A} = (a_{1}, a_{2})_{\Vec{r}}$ with $\Vec{r} = (r_{1}, r_{2})$. It is easy to see, in this case, $$S(\mathbb{A}) \supseteq \{ia_{1} : i \in [1,r_{1}]\} \cup \{(r_{1}-1)a_{1} + ia_{2} : i \in [1,r_{2}]\} \cup \{ r_{1}a_{1} + ia_{2} : i \in [1,r_{2}]\},$$  
	where the three sets on the right hand side are pairwise disjoint. Therefore (\ref{Equation-S-vec-A-1}) holds for $k=1, 2$. Assume that $k \geq 3$ and (\ref{Equation-S-vec-A-1}) holds for all sequences whose number of distinct terms is less than $k$. Set $\mathbb{B} = \{a_{1}, a_{2}, \ldots, a_{k-2}\}_{\Vec{s}}$ with $\Vec{s} = (r_{1}, r_{2}, \ldots, r_{k-2})$. Then $2^{\wedge}(A \cup \{0\})$ and $S(\mathbb{B}) + a_{k-1} + a_{k}$ are two disjoint subsets of $S(\mathbb{A})$. For $1\leq i \leq r_{k-1}-1$ and $1\leq j \leq k-2$, define
	\[s_{i,j} = \sum_{t=1, t\neq k-j-1}^{k-2} r_{t}a_{t} + (r_{k-j-1}-1)a_{k-j-1} + (i+1)a_{k-1} + a_{k},\] and 
	\[s_{i,k-1} = \sum_{t=1}^{k-2} r_{t}a_{t} + (i+1)a_{k-1} + a_{k}.\]
	Similarly, for $1\leq i \leq r_{k}-1$ and $1\leq j \leq k-1$, define
	\[u_{i,j} = \sum_{t=1, t\neq k-j}^{k-1} r_{t}a_{t} + (r_{k-j}-1)a_{k-j} + (i+1)a_{k},\] and 
	\[u_{i,k} = \sum_{t=1}^{k-1} r_{t}a_{t} + (i+1)a_{k}.\]
	It is easy to see that \[s_{i,1}< s_{i,2} < \cdots < s_{i,k-2} < s_{i,k-1} < s_{i+1,1},\]
	\[s_{r_{k-1}-1,k-1} < u_{1,1},\]
	and \[u_{i,1} < u_{i,2}< \cdots < u_{i,k-1} < u_{i,k} < u_{i+1,1}.\] 
	Therefore, the elements $s_{i, j}$ and $u_{i,j}$ are all distinct, all are in the set $S(\mathbb{A})$, and bigger than the elements of $2^{\wedge}(A \cup \{0\})$ and $S(\mathbb{B}) + a_{k-1} + a_{k}$. Note also that  $s_{i,j}$ is not defined for $r_{k-1}=1$ and  $u_{i,j}$ are not defined for  $r_{k}=1$. By Theorem \ref{Direct and Inverse Thm Nathanson restricted} and the induction hypothesis we get 
	\begin{align}\label{Equation-S-vec-A-2}
		\left|S(\mathbb{A})\right| 
		&\geq \left|2^{\wedge}(A \cup \{0\}) \right| + \left|S(\mathbb{B})+a_{k-1}+a_{k}\right| + \left|\bigcup_{i=1}^{r_{k-1}-1} \bigcup_{j=1}^{k-1} s_{i,j} \right| + \left|\bigcup_{i=1}^{r_{k}-1} \bigcup_{j=1}^{k} u_{i,j} \right| \nonumber\\
		&= \left|2^{\wedge}(A \cup \{0\}) \right| + \left|S(\mathbb{B})\right| + \sum_{i=1}^{r_{k-1}-1} \sum_{j=1}^{k-1} 1 + \sum_{i=1}^{r_{k}-1} \sum_{j=1}^{k} 1 \nonumber\\
		& \geq 2(k+1)-3 + \sum_{i=1}^{k-2} ir_{i} + (k-1)(r_{k-1}-1) + k(r_{k}-1) \nonumber\\
		& = \sum_{i=1}^{k}ir_{i}.
	\end{align}
	
	Now suppose that $k \geq 4$ and $\left|S(\mathbb{A}) \right| = \sum_{i=1}^{k}ir_{i}$. Then from (\ref{Equation-S-vec-A-2}) it follows that $\left| 2^{\wedge}(A \cup \{0\}) \right| = 2(k+1) - 3$. Theorem \ref{Direct and Inverse Thm Nathanson restricted} implies that $A\cup\{0\}$ is an arithmetic progression. Hence $\mathbb{A} = a_{1} \ast [1,k]_{\Vec{r}}$.
\end{proof}

\begin{theorem}\label{Freiman-theorem-1-for-vector-r}
	Let $\mathbb{A} = \{a_{1},a_{2}, \ldots, a_{k}\}_{\vec{r}}$ be a sequence of positive integers such that $a_{1}< a_{2}< \cdots < a_{k}$, $\vec{r} = (r_{1}, r_{2}, \ldots, r_{k})$ and $r_{i} \geq 1$ for all $i \in [1,k]$. Let $d(A) = 1$. Then 
	\begin{align*}
		\left|S(\mathbb{A})\right| 
		& \geq    
		\begin{cases}
			\sum_{i=1}^{k}ir_{i} + a_{k}-k& \text{ if } a_{k} \leq 2k-3 \\
			\sum_{i=1}^{k}ir_{i} + \theta(k+1)-k-4 & \text{ if }  a_{k} \geq 2k-2.
		\end{cases}
	\end{align*}
\end{theorem}

\begin{proof}	
	Set $\mathbb{B} = \{a_{1}, a_{2}, \ldots, a_{k-2}\}_{\Vec{s}}$ with $\Vec{s} = (r_{1}, r_{2}, \ldots, r_{k-2})$. From (\ref{Equation-S-vec-A-2}) we have 
	\begin{align*}
		\left|S(\mathbb{A})\right| 
		&\geq \left|2^{\wedge}(A \cup \{0\}) \right| + \left| S(\mathbb{B}) \right| + (k-1)(r_{k-1}-1) + k(r_{k}-1).
	\end{align*}
	Applying Theorem \ref{Lev restricted thm} on $A \cup \{0\}$ and Lemma \ref{Inverse-lemma-for-vector-r} on $\mathbb{B}$, we get 
	\begin{align*}
		&\left|S(\mathbb{A})\right| \\
		& \geq  
		\begin{cases}
			a_{k} + k-1 + \sum_{i=1}^{k-2}ir_{i} + (k-1)(r_{k-1}-1) + k(r_{k}-1)& \text{ if } a_{k} \leq 2(k+1)-5 \\
			(\theta +1)(k+1)-6 + \sum_{i=1}^{k-2}ir_{i} + (k-1)(r_{k-1}-1) + k(r_{k}-1)  & \text{ if }  a_{k} \geq 2(k+1)-4,
		\end{cases}\\
		& \geq  
		\begin{cases}
			\sum_{i=1}^{k}ir_{i} + a_{k}-k& \text{ if } a_{k} \leq 2k-3 \\
			\sum_{i=1}^{k}ir_{i} + \theta(k+1)-k-4 & \text{ if }  a_{k} \geq 2k-2.
		\end{cases}
	\end{align*}
	This completes the proof of the theorem. 
\end{proof}		

In the next theorem, we prove an improved bound for $\left|S(\mathbb{A}) \right|$ than Theorem \ref{Freiman-theorem-1-for-vector-r}, when every element of $\mathbb{A}$ appears at least twice in $\mathbb{A}$. Before that, we prove the following lemma which is crucial for our next theorem. The following lemma is same as Lemma \ref{Inverse-lemma-for-vector-r} but with $r_{i} \geq 2$ for all $i \in [1,k]$.  

\begin{lemma}\label{Direct-lemma-for-vector-r}\textup{\cite[Theorem 3.1, Theorem 3.2]{MistriPandey2015}} Let $\mathbb{A} = \{a_{1},a_{2}, \ldots, a_{k}\}_{\Vec{r}}$ be a sequence of positive integers such that $a_{1}< a_{2}< \cdots < a_{k}$, $\vec{r} = (r_{1}, r_{2}, \ldots, r_{k})$ and $r_{i} \geq 2$ for all $i \in [1,k]$. Then 
	\begin{equation}\label{Equation-S-vec-A-3}
		\left|S(\mathbb{A}) \right| \geq \sum_{i=1}^{k}ir_{i}.
	\end{equation} 
	
	Moreover, if $k \geq 4$ and $\left|S(\mathbb{A}) \right| = \sum_{i=1}^{k}ir_{i}$, then $\mathbb{A} = d\ast [1,k]_{\Vec{r}}$ for some positive integer $d$. 		
\end{lemma}

\begin{proof}
	Let $r := \min \{r_{1}, r_{2}, \ldots, r_{k}\}$. Clearly, (\ref{Equation-S-vec-A-3}) holds for $k=1$, as in this case $S(\mathbb{A}) = \{a, 2a, \ldots, ra\}$ if $A=(a)_{r}$. Assume that $k \geq 2$ and (\ref{Equation-S-vec-A-3}) holds for all sets that have less than $k$ distinct integers. Set $\mathbb{B}^{\prime} = \{a_{1}, a_{2}, \ldots, a_{k-1}\}_{\Vec{t}}$ with $\Vec{t} = (r_{1}, r_{2}, \ldots, r_{k-1})$. If $r_{k}=r$, then $r(A \cup \{0\})\setminus\{0\}$ and $S(\mathbb{B}^{\prime}) + ra_{k}$ are two disjoint subsets of $S(\mathbb{A})$. Thus, by Theorem \ref{Direct and Inverse Thm Nathanson} and the induction hypothesis, we get 
	\begin{align}\label{Equation-S-vec-A-4}
		\left| S(\mathbb{A}) \right| 
		\geq \left| r(A \cup \{0\}) \setminus \{0\} \right| + \left| S(\mathbb{B}^{\prime}) + ra_{k} \right| 
		\geq r(k+1) - r + \sum_{i=1}^{k-1}ir_{i} 
		= \sum_{i=1}^{k}ir_{i}.
	\end{align}
	Thus, (\ref{Equation-S-vec-A-3}) holds when $r_{k}=r$. If $r_{k} > r$, for $1\leq i \leq r_{k}-r$ and $1\leq j \leq k-1$, define
	\[v_{i,j} = \sum_{t=1, t\neq k-j}^{k-1}r_{t}a_{t} + (r_{k-j}-1)a_{k-j} + (r+i)a_{k},\] and 
	\[v_{i,k} = \sum_{t=1}^{k-1}r_{t}a_{t} + (r+i)a_{k}.\]
	Then \[v_{i,1} < v_{i,2} < \cdots < v_{i,k-1} < v_{i,k} < v_{i+1,1}.\]
	Therefore, the elements $v_{i, j}$ are all distinct, all are in the set $S(\mathbb{A})$, and bigger than the elements of $r(A \cup \{0\})$ and $S(\mathbb{B}^{\prime}) + ra_{k}$. Therefore, by Theorem \ref{Direct and Inverse Thm Nathanson} and the induction hypothesis we get 
	\begin{align}\label{Equation-S-vec-A-5}
		\left| S(\mathbb{A}) \right| 
		&\geq \left| r(A \cup \{0\}) \setminus \{0\} \right| + \left|  S(\mathbb{B}^{\prime}) + ra_{k} \right| + \left|\bigcup_{i=1}^{r_{k}-r}\bigcup_{j=1}^{k} v_{i,j} \right| \nonumber\\
		&\geq \left| r(A \cup \{0\}) \right| - 1 + \left|  S(\mathbb{B}^{\prime}) \right| + \sum_{i=1}^{r_k-r} \sum_{j=1}^{k} 1 \nonumber\\
		& \geq r(k+1)-r+ \sum_{i=1}^{k-1}ir_{i} + k(r_{k}-r) \nonumber\\
		&= \sum_{i=1}^{k}ir_{i}.
	\end{align}
	
	Now suppose that $r \geq 2$ and $\left|S(\mathbb{A}) \right| = \sum_{i=1}^{k}ir_{i}$. Then, from (\ref{Equation-S-vec-A-4}) and (\ref{Equation-S-vec-A-5}), we get that $\left| r(A \cup \{0\}) \right| = r(k+1) - r + 1$. Theorem \ref{Direct and Inverse Thm Nathanson} implies that $A\cup\{0\}$ is an arithmetic progression. Hence $\mathbb{A} = a_{1} \ast [1,k]_{\Vec{r}}$.
\end{proof}

\begin{theorem}\label{Freiman-theorem-2-for-vector-r}
	Let $\mathbb{A} = \{a_{1},a_{2}, \ldots, a_{k}\}_{\vec{r}}$ be a sequence of positive integers such that $a_{1}< a_{2}< \cdots < a_{k}$, $\vec{r} = (r_{1}, r_{2}, \ldots, r_{k})$ and $r_{i} \geq 2$ for all $i \in [1,k]$. Let $d(A) = 1$ and $\min \{r_1, r_2, \ldots, r_k\} = r $. Then  
	\[\left|S(\mathbb{A}) \right| \geq 
	|(r-1)(A\cup\{0\})| + \min \{ a_{k}, r(k-1)+1\} - 1 + \sum_{i=1}^{k-1}ir_{i} + k(r_{k}-r).\]
\end{theorem}

\begin{proof}	
	Set $\mathbb{B}^{\prime} = \{a_{1}, a_{2}, \ldots, a_{k-1}\}_{\Vec{t}}$ with $\Vec{t} = (r_{1}, r_{2}, \ldots, r_{k-1})$. Then from (\ref{Equation-S-vec-A-4}) and (\ref{Equation-S-vec-A-5}), we have 
	\begin{align*}
		\left| S(\mathbb{A}) \right| 
		&\geq \left| r(A \cup \{0\}) \setminus \{0\} \right| + \left|  S(\mathbb{B}^{\prime}) \right| + k(r_{k}-r).
	\end{align*}
	Applying Theorem \ref{Theorem Lev 1996} on $A \cup \{0\}$ and Lemma \ref{Direct-lemma-for-vector-r} on $\mathbb{B}^{\prime}$, we obtain
	\begin{align*}
		\left| S(\mathbb{A}) \right| 
		&\geq |(r-1)(A \cup \{0\})| + \min \{ a_{k}, r(k-1)+1\} -1 + \sum_{i=1}^{k-1}ir_{i} + k(r_{k}-r).
	\end{align*}
	This completes the proof of the theorem. 
\end{proof}

In the following theorem we prove a result similar to Theorem \ref{Freiman-theorem-1-for-vector-r} and Theorem \ref{Freiman-theorem-2-for-vector-r}, when the sequence $\mathbb{A}$ contains nonnegative integers with $0 \in \mathbb{A}$.  

\begin{theorem}\label{Freiman-theorem-for-S(A)-zero-1}
	Let $\mathbb{A} = \{a_{0},a_{1}, \ldots, a_{k-1}\}_{\vec{r}}$ be a sequence of nonnegative integers such that $0 = a_{0}< a_{1}< \cdots < a_{k-1}$, $\vec{r} = (r_{0}, r_{1}, \ldots, r_{k-1})$ and $r_{i} \geq 1$ for all $i \in [0,k-1]$. Let $d(A) = 1$ and $\min \{r_{1}, r_{2}, \ldots, r_{k-1}\} = r $. 
	\begin{enumerate}
		\item If $r\geq 2$, then 
		\[\left|S(\mathbb{A}) \right| \geq 
		|(r-1)A| + \min \{ a_{k-1}, r(k-2)+1\} + \sum_{i=1}^{k-1} (i-1)r_{i-1} + (k-1)(r_{k-1}-r).\]
		
		\item If $r=1$, then 
		\begin{equation*}
			\left|S(\mathbb{A})\right| 
			\geq    
			\begin{cases}
				a_{k-1}-k+2 + \sum_{i=1}^{k} (i-1)r_{i-1} & \text{ if }     a_{k-1} \leq 2k-5 \\
				\theta k-k-2 + \sum_{i=1}^{k} (i-1)r_{i-1} & \text{ if }  a_{k-1} \geq 2k-4.
			\end{cases}
		\end{equation*}
	\end{enumerate}
\end{theorem}

\begin{proof}
	Let $\mathbb{B}^{\prime\prime} := \{a_{1}, a_{2}, \ldots, a_{k-1}\}_{\vec{v}}$ with $\vec{v} := (r_{1}, r_{2}, \ldots, r_{k-1})$. Then $d(B^{\prime\prime})=1$ and $S(\mathbb{A}) = S(\mathbb{B}^{\prime\prime}) \cup\{0\}$. If $r\geq 2$, then applying Theorem \ref{Freiman-theorem-2-for-vector-r} on $\mathbb{B}^{\prime\prime}$, we obtain 
	\begin{align*}
		\left| S(\mathbb{A}) \right| 
		&\geq \left| S(\mathbb{B}^{\prime\prime}) \right| + 1 \\
		&\geq |(r-1)A| + \min \{ a_{k-1}, r(k-2)+1\} - 1 + \sum_{i=1}^{k-2}ir_{i} + (k-1)(r_{k-1}-r) + 1 \\
		&\geq |(r-1)A| + \min \{ a_{k-1}, r(k-2)+1\} + \sum_{i=1}^{k-1} (i-1)r_{i-1} + (k-1)(r_{k-1}-r).
	\end{align*}
	If $r=1$, then by Theorem \ref{Freiman-theorem-1-for-vector-r}, we have 
	\begin{align*}
		\left| S(\mathbb{A}) \right| 
		&\geq \left| S(\mathbb{B}^{\prime\prime}) \right| + 1 \\
		& \geq    
		\begin{cases}
			a_{k-1}-k+1 + \sum_{i=1}^{k-1} ir_{i}  + 1 & \text{ if } a_{k-1} \leq 2(k-1)-3 \\
			\theta k -k- 3 + \sum_{i=1}^{k-1} ir_{i}  + 1 & \text{ if }  a_{k-1} \geq 2(k-1)-2,
		\end{cases} \\
		& =    
		\begin{cases}
			a_{k-1}-k+2 + \sum_{i=1}^{k} (i-1)r_{i-1} & \text{ if } a_{k-1} \leq 2k-5 \\
			\theta k-k-2 + \sum_{i=1}^{k} (i-1)r_{i-1} & \text{ if }  a_{k-1} \geq 2k-4.
		\end{cases}
	\end{align*}
	This completes the proof of the theorem. 
\end{proof}


\section{Freiman's theorem for $\alpha$-subset sum}\label{section-alpha-subset-sum}

In this section, we prove Freiman-like theorems for $S_{\alpha}(A)$, when the set $A$ contains positive integers and when the set $A$ contains nonnegative integers with $0 \in A$, in Theorem \ref{Theorem-S_alpha-1} and Theorem \ref{Freiman-Theorem-S_alpha-zero}, respectively. To prove our results, we define 
\[ S_{1}^{\alpha}(A) := \left\{s(B): B \subseteq A, 1 \leq |B| \leq |A|-\alpha\right\}. \] 
Then $S_{\alpha}(A) = \sum_{a \in A} a - \left( S_{1}^{\alpha}(A) \cup \{0\} \right)$. Therefore, $|S_{\alpha}(A)| = |S_{1}^{\alpha}(A)|+1$ if $0 \notin A$ and $|S_{\alpha}(A)| = |S_{1}^{\alpha}(A)|$ if $0 \in A$.


\begin{theorem}\label{Theorem-S_alpha-1}
	Let $k \geq 3$. Let $A = \{a_{1},a_{2},\ldots,a_{k}\}$ be a set of $k$ positive integers such that $a_{1} < a_{2} < \cdots < a_{k}$ and $d(A)=1$. Let $\alpha \leq k-2$ be a positive integer. Then 		
	\[\left|S_{\alpha}(A)\right| 
	\geq 
	\begin{cases}
		a_{k} + \dfrac{k(k-1)}{2} - \dfrac{\alpha(\alpha+1)}{2} + 1 & \text{ if } a_{k} \leq 2k-3 \vspace{0.2cm}\\
		\theta(k+1) - 4 + \dfrac{k(k-1)}{2} - \dfrac{\alpha(\alpha+1)}{2} + 1 & \text{ if } a_{k} \geq 2k-2.
	\end{cases} \]
\end{theorem}

\begin{proof}
	Set $B = A \setminus \{a_{k-1}, a_{k}\}$. Then 
	\begin{equation*}
		2^{\wedge}(A\cup \{0\}) \cup  (S^{\alpha}_{1}(B) + a_{k-1} + a_{k}) \subset S_{1}^{\alpha}(A).
	\end{equation*}
	Here we are assuming that $S^{\alpha}_{1}(B) + a_{k-1} + a_{k} = \emptyset$ if $\alpha = |B|$. Observe that $2^{\wedge}(A\cup \{0\})$ and $S^{\alpha}_{1}(B) + a_{k-1} + a_{k}$ are disjoint. Thus
	\begin{align*}
		\left| S_{\alpha}(A) \right| 
		= \left| S^{\alpha}_{1}(A) \right| + 1 \geq \left| 2^{\wedge}(A\cup \{0\}) \right| + \left| S^{\alpha}_{1}(B)  \right| + 1 
		= \left| 2^{\wedge}(A\cup \{0\}) \right| + \left|  S_{\alpha}(B) \right|.
	\end{align*} 
	If $a_{k} \leq 2k-3 = 2(k+1)-5$, then applying Theorem \ref{Lev restricted thm} on $A\cup \{0\}$ and Theorem \ref{theorem-BP_for-subset-sums} on $B$, we obtain 
	\begin{align*}
		\left| S_{\alpha}(A) \right| 
		&\geq a_{k}+k-1 + \frac{(k-2)(k-1)}{2}- \frac{\alpha(\alpha+1)}{2} + 1 \\ 
		&= a_{k} + \dfrac{k(k-1)}{2} - \dfrac{\alpha(\alpha+1)}{2} + 1.
	\end{align*} 
	If $a_{k} \geq 2k-2 = 2(k+1)-4 $, then again by Theorem \ref{Lev restricted thm} and Theorem \ref{theorem-BP_for-subset-sums}, we get
	\begin{align*}
		\left| S_{\alpha}(A) \right| 
		&\geq (\theta+1)(k+1)-6 + \frac{(k-2)(k-1)}{2}- \frac{\alpha(\alpha+1)}{2} + 1 \\ 
		&= \theta(k+1) - 4 + \dfrac{k(k-1)}{2} - \dfrac{\alpha(\alpha+1)}{2} + 1.
	\end{align*} 
	This proves the theorem. 
\end{proof}


We also have the following theorem when the set $A$ has $0$ as an element. 

\begin{theorem}\label{Freiman-Theorem-S_alpha-zero}
	Let $k \geq 4$. Let $A = \{a_{0},a_{1},\ldots,a_{k-1}\}$ be a set of nonnegative integers such that $0 = a_{0} < a_{1} < \cdots < a_{k-1}$ and $d(A)=1$. Let $\alpha \leq k-2$ be a positive integer. Then   
	\[
	\left| S_{\alpha}(A) \right|  
	\geq 
	\begin{cases}
		a_{k-1} + \dfrac{(k-1)(k-2)}{2} - \dfrac{\alpha(\alpha-1)}{2} + 2 & \text{ if } a_{k-1} \leq 2k-5 \vspace{0.2cm}\\ 
		\theta k -4 + \dfrac{(k-1)(k-2)}{2} - \dfrac{\alpha(\alpha-1)}{2} + 2 & \text{ if } a_{k-1} \geq 2k-4.
	\end{cases} 
	\]
\end{theorem}

\begin{proof}  
	Set $B^{\prime} = \{a_{1}, a_{2}, \ldots, a_{k-1}\}$. Then $B^{\prime}$ is a set of $k-1$ positive integers, $d(B^{\prime})=1$ and $S_{1}^{\alpha}(A) = S_{1}^{\alpha-1}(B^{\prime}) \cup \{0\}$. Then from Theorem \ref{Theorem-S_alpha-1}, it follows that  
	\begin{align*}
		\left| S_{\alpha}(A) \right| 
		&= \left| S_{1}^{\alpha}(A) \right|\\ 
		&= \left| S_{1}^{\alpha-1}(B^{\prime}) \right| +1 \\
		&\geq 
		\begin{cases}
			a_{k-1} + \dfrac{(k-1)(k-2)}{2} - \dfrac{\alpha(\alpha-1)}{2} + 2 
			& \text{ if } a_{k-1} \leq 2k-5 \vspace{0.2cm}\\ 
			\theta k - 4 + \dfrac{(k-1)(k-2)}{2} - \dfrac{\alpha(\alpha-1)}{2} + 2
			& \text{ if } a_{k-1} \geq 2k-4.
		\end{cases} 
	\end{align*} 
	This proves the theorem. 
\end{proof}

\section{Freiman's theorem for $\alpha$-subsequence sum}\label{section-alpha-subsequence-sum}

Let $\mathbb{A} = \{a_{1},a_{2},\ldots,a_{k}\}_{\vec{r}}$ be a sequence of positive integers, where $\vec{r} = (r_{1}, r_{2}, \ldots, r_{k})$ with $r_{i} \geq 1$ for all $i\in [1,k]$. Let $\min\{r_{1}, r_{2}, \ldots, r_{k}\} = r$, and let $\alpha \leq \sum_{i=1}^{k} r_{i} -2$ be a positive integer. In this section, we prove Freiman's $3k-4$-like results for $S_{\alpha}(\mathbb{A})$. The proofs are quite similar to the ones in Section 4, however, in this section, the proofs are more involved and depend heavily on $\alpha$. In Theorem \ref{the-last-value-of-alpha}, we assume that $\alpha=\sum_{i=1}^{k} r_{i} -2$. Then, in Theorems \ref{Freiman-Theorem-for-S_alpha-sequence-1} and \ref{Freiman-Theorem-for-S_alpha-sequence-2}, we consider the case that $r \geq 2$ and $\alpha < \sum_{i=1}^{k} r_{i} -2$. Further, in Theorems \ref{Freiman-Theorem-for-S_alpha-sequence-3} and \ref{Freiman-Theorem-for-S_alpha-sequence-4}, we assume that $r=1$ and $\alpha < \sum_{i=1}^{k} r_{i} -2$. In Theorem \ref{Freiman-Theorem-for-S_alpha-sequence-3}, we consider all possible cases with $r=1$, except the one that $r_{k-1} = 1$ and $r_{k} \neq 1$, which we deal in Theorem \ref{Freiman-Theorem-for-S_alpha-sequence-4}. In all the above-mentioned theorems the sequence $\mathbb{A}$ contains positive integers. We prove similar results in Theorems \ref{the-last-value-of-alpha-zero}, \ref{Freiman-Theorem-for-S_alpha-sequence-5} and \ref{Freiman-Theorem-for-S_alpha-sequence-6}, when the sequence $\mathbb{A}$ contains nonnegative integers with $0 \in \mathbb{A}$.     

Before proceeding to the results of this section, we first define 
\[ S_{1}^{\alpha}(\mathbb{A}) := \left\{s(\mathbb{B}): \mathbb{B} \subseteq \mathbb{A}, 1 \leq |\mathbb{B}| \leq \sum_{i=1}^{k}r_{i}-\alpha\right\}. \] 
Then $S_{\alpha}(\mathbb{A}) = \sum_{a \in \mathbb{A}} a - \left( S_{1}^{\alpha}(\mathbb{A}) \cup \{0\} \right)$. Therefore, $|S_{\alpha}(\mathbb{A})| = |S_{1}^{\alpha}(\mathbb{A})|+1$ if $0 \notin \mathbb{A}$ and $|S_{\alpha}(\mathbb{A})| = |S_{1}^{\alpha}(\mathbb{A})|$ if $0 \in \mathbb{A}$. 

\begin{theorem}\label{the-last-value-of-alpha}
	Let $k \geq 3$. Let $\mathbb{A} = \{a_{1}, a_{2},\ldots, a_{k}\}_{\vec{r}}$ be a sequence of positive integers such that $a_{1} < a_{2} < \cdots < a_{k}$, $\vec{r} = (r_{1}, r_{2}, \ldots, r_{k})$ and $r_{i} \geq 1$ for all $i\in [1,k]$. Let $r=\min\{r_{1}, r_{2}, \ldots, r_{k}\}$. Let $\alpha = \sum_{i=1}^{k} r_{i}-2$ and $d(A)=1$. If $r=1$, then 
	\begin{equation*}
		\left|S_{\alpha}(\mathbb{A})\right| \geq \begin{cases}
			a_{k} + k  & \text{ if } a_{k} \leq 2k-3\\
			(\theta + 1)(k+1) - 4 & \text{ if } a_{k} \geq 2k-2.
		\end{cases}
	\end{equation*}
	If $r \geq 2$, then 
	\begin{equation*}
		\left|S_{\alpha}(\mathbb{A})\right| \geq \begin{cases}
			a_{k} + k +1 &  \text{ if } a_{k} \leq 2k-1\\
			3k & \text{ if } a_{k} \geq 2k.
		\end{cases} 
	\end{equation*}
\end{theorem}

\begin{proof}
	If $r=1$, then \[ S^{\alpha}_{1} (\mathbb{A}) = 2^{\wedge} (A\cup \{0\}). \]
	Therefore, by Theorem \ref{Lev restricted thm}, we get 
	\begin{align*}
		\left| S_{\alpha}(\mathbb{A}) \right| 
		= \left| S^{\alpha}_{1}(\mathbb{A}) \right| + 1 
		\geq \begin{cases}
			a_{k} + k & \text{ if } a_{k} \leq 2k-3\\
			(\theta + 1)(k+1) - 4 & \text{ if } a_{k} \geq 2k-2.
		\end{cases}
	\end{align*}
	If $r \geq 2$, then \[S^{\alpha}_{1} (\mathbb{A}) = 2(A\cup \{0\}) \setminus \{0\}.\] 
	Therefore, by  Theorem \ref{Theorem Freiman II}, we get 
	\begin{align*}
		\left| S_{\alpha}(\mathbb{A}) \right| 
		= \left| S^{\alpha}_{1}(\mathbb{A}) \right| + 1 
		\geq \begin{cases}
			a_{k} + k +1 & \text{ if } a_{k} \leq 2k-1\\
			3k & \text{ if } a_{k} \geq 2k.
		\end{cases}
	\end{align*}
	
\end{proof}

\begin{theorem}\label{Freiman-Theorem-for-S_alpha-sequence-1}
	Let $k \geq 3$. Let $\mathbb{A} = \{a_{1}, a_{2},\ldots, a_{k}\}_{\vec{r}}$ be a sequence of positive integers such that $a_{1} < a_{2} < \cdots < a_{k}$, $\vec{r} = (r_{1}, r_{2}, \ldots, r_{k})$ and $r_{i} \geq 2$ for all $i\in [1,k]$. Let $\min\{r_{1}, r_{2}, \ldots, r_{k}\} = r$ and $d(A)=1$. Let $ \alpha < \sum_{i=1}^{k} r_{i}-r$ be a positive integer. Then there exists an integer $m \in [1,k]$ such that $\sum_{i=1}^{m-1} r_{i} \leq \alpha < \sum_{i=1}^{m} r_{i}$ and 
	\begin{equation*}
		\left|S_{\alpha}(\mathbb{A})\right| \geq |(r-1)(A \cup \{0\})| + \min \{ a_{k}, r(k-1)+1\} + \sum_{i=1}^{k-1} ir_{i} - \sum_{i=1}^{m} ir_{i} + m\left(\sum_{i=1}^{m} r_{i}-\alpha\right) + k(r_{k}-r). 
	\end{equation*}
\end{theorem}

\begin{proof}
	Set $\mathbb{B}_{1} = \{a_{1}, a_{2}, \ldots, a_{k-1}, a_{k}\}_{\vec{s_{1}}}$ with $\vec{s_{1}} = (r_{1}, r_{2}, \ldots, r_{k-1}, r_{k}-r)$. Then
	\begin{equation*}
		(r(A\cup \{0\}) \setminus \{0\}) \cup  (S^{\alpha}_{1}(\mathbb{B}_{1}) + ra_{k}) \subset S_{1}^{\alpha}(\mathbb{A}),
	\end{equation*}
	where $(r(A\cup \{0\}) \setminus \{0\}) \cap (S^{\alpha}_{1}(\mathbb{B}_{1}) + ra_{k}) = \emptyset$. Therefore
	\begin{align*}
		\left| S_{\alpha}(\mathbb{A}) \right| 
		= \left| S^{\alpha}_{1}(\mathbb{A}) \right| + 1 
		\geq \left| r(A\cup \{0\}) \setminus \{0\} \right| + \left| S^{\alpha}_{1} (\mathbb{B}_{1}) \right| + 1 
		= \left| r(A\cup \{0\}) \right| + \left| S_{\alpha}(\mathbb{B}_{1}) \right| -1. 	
	\end{align*} 
	If $m \leq k-1$, then applying Theorem \ref{Theorem Lev 1996} on $A\cup \{0\}$ and Theorem \ref{theorem-BP_for-subsequence-sums} on $\mathbb{B}_{1}$, we obtain
	\begin{align*}
		\left| S_{\alpha}(\mathbb{A}) \right| 
		&\geq 
		|(r-1)(A \cup \{0\})| + \min \{ a_{k}, r(k-1)+1\} -1 + \sum_{i=1}^{k-1} ir_{i} + k(r_{k}-r) - \sum_{i=1}^{m} ir_{i} \\
		&\quad+ m\left(\sum_{i=1}^{m} r_{i}-\alpha\right) + 1\\
		&= |(r-1)(A \cup \{0\})| + \min \{ a_{k}, r(k-1)+1\} + \sum_{i=1}^{k-1} ir_{i} - \sum_{i=1}^{m} ir_{i} + m\left(\sum_{i=1}^{m} r_{i}-\alpha\right) \\
		&\quad+ k(r_{k}-r). 	
	\end{align*} 
	If $m = k$, then again applying Theorem \ref{Theorem Lev 1996} on $A\cup \{0\}$ and Theorem \ref{theorem-BP_for-subsequence-sums} on $\mathbb{B}_{1}$, we obtain
	\begin{align*}
		\left| S_{\alpha}(\mathbb{A}) \right| 
		&\geq 
		|(r-1)(A \cup \{0\})| + \min \{ a_{k}, r(k-1)+1\} + k\left(\sum_{i=1}^{k-1} r_{i} + r_{k}-r-\alpha\right). 
	\end{align*} 
	This proves the theorem. 
\end{proof}

In the following theorem, we prove a similar result  for the remaining values of $\alpha$, i.e., $\sum_{i=1}^{k} r_{i}-r \leq \alpha < \sum_{i=1}^{k} r_{i}-2$.

\begin{theorem}\label{Freiman-Theorem-for-S_alpha-sequence-2}
	Let $k \geq 3$. Let $\mathbb{A} = \{a_{1}, a_{2},\ldots, a_{k}\}_{\vec{r}}$ be a sequence of positive integers such that $ a_{1} < a_{2} < \cdots < a_{k}$, $\vec{r} = (r_{1}, r_{2}, \ldots, r_{k})$ and $r_{i} \geq 2$ for all $i\in [1,k]$. Let $\min\{r_{1}, r_{2}, \ldots, r_{k}\} = r$ and $d(A)=1$. Let $\sum_{i=1}^{k} r_{i}-r \leq \alpha < \sum_{i=1}^{k} r_{i}-2$ be a positive integer. Then 
	\begin{equation*}
		\left|S_{\alpha}(\mathbb{A})\right| 
		\geq \begin{cases}
			a_{k}-k+2 + k\left(\sum_{i=1}^{k} r_{i}-\alpha\right) & \text{ if } a_{k} \leq 2k-1 \vspace{0.2cm}\\
			k+1 + k\left(\sum_{i=1}^{k} r_{i}-\alpha\right) & \text{ if } a_{k} \geq 2k.
		\end{cases}	
	\end{equation*}
\end{theorem}

\begin{proof}
	Set $\mathbb{B}_{2} = \{a_{1}, a_{2}, \ldots, a_{k-1}, a_{k}\}_{\vec{s_{2}}}$ with $\vec{s_{2}} = (r_{1}, r_{2}, \ldots, r_{k-1}, r_{k}-2)$. Then
	\begin{equation*}
		(2(A\cup \{0\}) \setminus \{0\}) \cup  (S^{\alpha}_{1}(\mathbb{B}_{2}) + 2a_{k}) \subset S_{1}^{\alpha}(\mathbb{A}),
	\end{equation*}
	where $(2(A\cup \{0\}) \setminus \{0\}) \cap (S^{\alpha}_{1}(\mathbb{B}_{2}) + 2a_{k}) = \emptyset$. Therefore,
	\begin{align*}
		\left| S_{\alpha}(\mathbb{A}) \right| 
		= \left| S^{\alpha}_{1}(\mathbb{A}) \right| + 1 
		\geq \left| 2(A\cup \{0\}) \setminus \{0\} \right| + \left| S^{\alpha}_{1}(\mathbb{B}_{2}) \right| + 1  
		= \left| 2(A\cup \{0\}) \right| + \left| S_{\alpha}(\mathbb{B}_{2}) \right|-1.	
	\end{align*} 
	Applying Theorem \ref{Theorem Freiman I} on $A\cup \{0\}$ and Theorem \ref{theorem-BP_for-subsequence-sums} on $\mathbb{B}_{2}$ we obtain
	\begin{align*}
		\left| S_{\alpha}(\mathbb{A}) \right| 
		&\geq \left| 2(A\cup \{0\}) \right| -1 + k\left(\sum_{i=1}^{k-1} r_{i} + r_{k}-2-\alpha\right)+ 1 \\
		&\geq 
		\begin{cases}
			a_{k}+k+1 + k\left(\sum_{i=1}^{k} r_{i}-\alpha\right) - 2k  & \text{ if } a_{k} \leq 2(k+1)-3 \vspace{0.2cm}\\
			3(k+1)-3 + k\left(\sum_{i=1}^{k} r_{i}-\alpha\right) - 2k  & \text{ if } a_{k} \geq 2(k+1)-2,
		\end{cases} \\
		&= 
		\begin{cases}
			a_{k}-k+1 + k\left(\sum_{i=1}^{k} r_{i}-\alpha\right) & \text{ if } a_{k} \leq 2k-1 \vspace{0.2cm}\\
			k + k\left(\sum_{i=1}^{k} r_{i}-\alpha\right) & \text{ if } a_{k} \geq 2k.
		\end{cases}		
	\end{align*}  
\end{proof}

The case $r=1$ is considered in the following two theorems.

\begin{theorem}\label{Freiman-Theorem-for-S_alpha-sequence-3}
	Let $k \geq 3$. Let $\mathbb{A} = \{a_{1}, a_{2},\ldots, a_{k}\}_{\vec{r}}$ be a sequence of positive integers such that $a_{1} < a_{2} < \cdots < a_{k}$, $\vec{r} = (r_{1}, r_{2}, \ldots, r_{k})$ and $r_{i} \geq 1$ for all $i\in [1,k]$. Let $r_{k-1} \neq 1$ and $r_{k} \neq 1$ or $r_{k-1} = r_{k} = 1$ or $r_{k-1} \neq 1$ and $r_{k} = 1$. Let $d(A)=1$. Let $ \alpha < \sum_{i=1}^{k} r_{i}-2$ be a positive integer. Then the following holds. 
	\begin{enumerate}
		\item[\upshape(1)] If $\alpha < \sum_{i=1}^{k-1} r_{i}-1$, then there exists an integer $m \in [1,k-1]$ such that $\sum_{i=1}^{m-1} r_{i} \leq \alpha < \sum_{i=1}^{m} r_{i} $ and 
		\begin{equation*}
			\left|S_{\alpha}(\mathbb{A})\right| \geq \begin{cases}
				a_{k}-k+1 + \sum_{i=1}^{k} ir_{i} - \sum_{i=1}^{m} ir_{i} + m\left(\sum_{i=1}^{m} r_{i}-\alpha\right) & \text{ if } a_{k} \leq 2k-3 \vspace{0.2cm}\\
				\theta (k+1)-k-3 + \sum_{i=1}^{k} ir_{i} - \sum_{i=1}^{m} ir_{i} + m\left(\sum_{i=1}^{m} r_{i}-\alpha\right) & \text{ if } a_{k} \geq 2k-2.
			\end{cases}	
		\end{equation*}
		
		\item[\upshape(2)] If $\sum_{i=1}^{k-1} r_{i}-1 \leq \alpha < \sum_{i=1}^{k} r_{i} -2$, then 
		\begin{equation*}
			\left|S_{\alpha}(\mathbb{A})\right| \geq 
			\begin{cases}
				a_{k}-k + k \left(\sum_{i=1}^{k} r_{i}-\alpha\right) & \text{ if } a_{k} \leq 2k-3 \vspace{0.2cm}\\
				\theta (k+1)-k-4 + k \left(\sum_{i=1}^{k} r_{i}-\alpha\right) & \text{ if } a_{k} \geq 2k-2.
			\end{cases}
		\end{equation*}
	\end{enumerate}
\end{theorem}

\begin{proof}
	Set $\mathbb{B}_{3} = \{a_{1}, a_{2}, \ldots, a_{k-1}, a_{k}\}_{\vec{s_{3}}}$ with $\vec{s_{3}} = (r_{1}, r_{2}, \ldots, r_{k-2}, r_{k-1}-1, r_{k}-1)$. Then
	\begin{equation*}
		2^{\wedge}(A\cup \{0\}) \cup  (S^{\alpha}_{1}(\mathbb{B}_{3}) + a_{k-1} + a_{k}) \subset S_{1}^{\alpha}(\mathbb{A}),
	\end{equation*}
	where $2^{\wedge}(A\cup \{0\}) \cap (S^{\alpha}_{1}(\mathbb{B}_{3}) + a_{k-1} + a_{k}) = \emptyset$. Therefore, 
	\begin{equation*}
		\left| S_{\alpha}(\mathbb{A}) \right|
		= \left| S^{\alpha}_{1}(\mathbb{A}) \right| + 1 
		\geq \left| 2^{\wedge}(A\cup \{0\}) \right| + \left| S^{\alpha}_{1}(\mathbb{B}_{3}) \right| + 1 
		= \left| 2^{\wedge}(A\cup \{0\}) \right| + \left| S_{\alpha}(\mathbb{B}_{3}) \right|.
	\end{equation*}   
	
	\textbf{Case I}	($r_{k-1} \geq 2$ and $r_{k} \geq 2$). If $\alpha < \sum_{i=1}^{k-2} r_{i}$, then $m \leq k-2$ for both $\mathbb{A}$ and $\mathbb{B}_{3}$. Applying Theorem \ref{Lev restricted thm} on $A\cup \{0\}$ and Theorem \ref{theorem-BP_for-subsequence-sums} on $\mathbb{B}_{3}$, we get
	\begin{align*}
		&\left| S_{\alpha}(\mathbb{A}) \right| \\ 
		&\geq \left| 2^{\wedge}(A\cup \{0\}) \right| + \left| S_{\alpha}(\mathbb{B}_{3}) \right|	\\
		&\geq 
		\begin{cases}
			a_{k}+k-1 + \sum_{i=1}^{k} ir_{i} -(k-1)-k - \sum_{i=1}^{m} ir_{i} + m\left(\sum_{i=1}^{m} r_{i}-\alpha\right) +1 & \text{ if } a_{k} \leq 2k-3 \vspace{0.2cm}\\
			(\theta+1)(k+1)-6 + \sum_{i=1}^{k} ir_{i} - (k-1)-k - \sum_{i=1}^{m} ir_{i} + m\left(\sum_{i=1}^{m} r_{i}-\alpha\right)+1 & \text{ if } a_{k} \geq 2k-2,
		\end{cases}	\\
		&= 
		\begin{cases}
			a_{k}-k+1 + \sum_{i=1}^{k} ir_{i} - \sum_{i=1}^{m} ir_{i} + m\left(\sum_{i=1}^{m} r_{i}-\alpha\right) & \text{ if } a_{k} \leq 2k-3 \vspace{0.2cm}\\
			\theta (k+1)-k-3 + \sum_{i=1}^{k} ir_{i} - \sum_{i=1}^{m} ir_{i} + m\left(\sum_{i=1}^{m} r_{i}-\alpha\right) & \text{ if } a_{k} \geq 2k-2.
		\end{cases}	
	\end{align*}
	
	If $\sum_{i=1}^{k-2} r_{i} \leq \alpha < \sum_{i=1}^{k-1} r_{i}-1$, then $m=k-1$ for both $\mathbb{A}$ and $\mathbb{B}_{3}$. Applying Theorem \ref{Lev restricted thm} on $A\cup \{0\}$ and Theorem \ref{theorem-BP_for-subsequence-sums} on $\mathbb{B}_{3}$, we get
	\begin{align*}
		&\left| S_{\alpha}(\mathbb{A}) \right| \\ 
		&\geq \left| 2^{\wedge}(A\cup \{0\}) \right| + \left| S_{\alpha}(\mathbb{B}_{3}) \right|	\\
		&\geq 
		\begin{cases}
			&a_{k}+k-1 + \sum_{i=1}^{k} ir_{i}-(k-1)-k - \left(\sum_{i=1}^{k-1} ir_{i}-(k-1)\right) + (k-1) \left(\sum_{i=1}^{k-1} r_{i}-1-\alpha\right)+1\\ & \text{ if } a_{k} \leq 2k-3 \vspace{0.2cm}\\
			&(\theta+1)(k+1)-6 + \sum_{i=1}^{k} ir_{i}-(k-1)-k - \left(\sum_{i=1}^{k-1} ir_{i}-(k-1)\right) + (k-1) \left(\sum_{i=1}^{k-1} r_{i}-1-\alpha\right)+1\\ & \text{ if } a_{k} \geq 2k-2,
		\end{cases}	\\
		&= 
		\begin{cases}
			a_{k}-k+1 + \sum_{i=1}^{k} ir_{i} - \sum_{i=1}^{k-1} ir_{i} + (k-1)\left(\sum_{i=1}^{k-1} r_{i}-\alpha\right) & \text{ if } a_{k} \leq 2k-3 \vspace{0.2cm}\\
			\theta (k+1)-k-3 + \sum_{i=1}^{k} ir_{i} - \sum_{i=1}^{k-1} ir_{i} + (k-1)\left(\sum_{i=1}^{k-1} r_{i}-\alpha\right) & \text{ if } a_{k} \geq 2k-2.
		\end{cases}	
	\end{align*}
	
	If $\sum_{i=1}^{k-1} r_{i}-1 \leq \alpha < \sum_{i=1}^{k} r_{i} -2$, then applying Theorem \ref{Lev restricted thm} on $A\cup \{0\}$ and Theorem \ref{theorem-BP_for-subsequence-sums} on $\mathbb{B}_{3}$, we get
	\begin{align*}
		&\left| S_{\alpha}(\mathbb{A}) \right| \\ 
		&\geq \left| 2^{\wedge}(A\cup \{0\}) \right| + \left| S_{\alpha}(\mathbb{B}_{3}) \right|	\\
		&\geq 
		\begin{cases}
			a_{k}+k-1 + k \left(\sum_{i=1}^{k} r_{i}-2-\alpha\right) +1 & \text{ if } a_{k} \leq 2k-3 \vspace{0.2cm}\\
			(\theta+1)(k+1)-6 + k \left(\sum_{i=1}^{k} r_{i}-2-\alpha\right) +1 & \text{ if } a_{k} \geq 2k-2,
		\end{cases}	\\
		&= 
		\begin{cases}
			a_{k}-k + k \left(\sum_{i=1}^{k} r_{i}-\alpha\right) & \text{ if } a_{k} \leq 2k-3 \vspace{0.2cm}\\
			\theta (k+1)-k-4 + k \left(\sum_{i=1}^{k} r_{i}-\alpha\right) & \text{ if } a_{k} \geq 2k-2.
		\end{cases}
	\end{align*}
	
	\textbf{Case II} ($r_{k-1} \geq 2$ and $r_{k} = 1$). In this case, the sequence $\mathbb{B}_{3}$ has $k-1$ distinct elements and the vector $s_{3}$ has $k-1$ terms. If $\alpha < \sum_{i=1}^{k-2} r_{i}$, then $m \leq k-2$ for both $\mathbb{A}$ and $\mathbb{B}_{3}$. Applying Theorem \ref{Lev restricted thm} on $A\cup \{0\}$ and Theorem \ref{theorem-BP_for-subsequence-sums} on $\mathbb{B}_{3}$, we get
	\begin{align*}
		&\left| S_{\alpha}(\mathbb{A}) \right| \\ 
		&\geq \left| 2^{\wedge}(A\cup \{0\}) \right| + \left| S_{\alpha}(\mathbb{B}_{3}) \right|	\\
		&\geq 
		\begin{cases}
			a_{k}+k-1 + \sum_{i=1}^{k-2} ir_{i}+(k-1)(r_{k-1}-1) - \sum_{i=1}^{m} ir_{i} + m\left(\sum_{i=1}^{m} r_{i}-\alpha\right) +1 & \text{ if } a_{k} \leq 2k-3 \vspace{0.2cm}\\
			(\theta+1)(k+1)-6 + \sum_{i=1}^{k-2} ir_{i}+(k-1)(r_{k-1}-1) - \sum_{i=1}^{m} ir_{i} + m\left(\sum_{i=1}^{m} r_{i}-\alpha\right)+1 & \text{ if } a_{k} \geq 2k-2,
		\end{cases}	\\
		&= 
		\begin{cases}
			a_{k}-k+1 + \sum_{i=1}^{k} ir_{i} - \sum_{i=1}^{m} ir_{i} + m\left(\sum_{i=1}^{m} r_{i}-\alpha\right) & \text{ if } a_{k} \leq 2k-3 \vspace{0.2cm}\\
			\theta (k+1)-k-3 + \sum_{i=1}^{k} ir_{i} - \sum_{i=1}^{m} ir_{i} + m\left(\sum_{i=1}^{m} r_{i}-\alpha\right) & \text{ if } a_{k} \geq 2k-2.
		\end{cases}	
	\end{align*}
	If $\sum_{i=1}^{k-2} r_{i} \leq \alpha < \sum_{i=1}^{k-1} r_{i}-1 = \sum_{i=1}^{k} r_{i}-2$, then $m=k-1$ for both $\mathbb{A}$ and $\mathbb{B}_{3}$. Applying Theorem \ref{Lev restricted thm} on $A\cup \{0\}$ and Theorem \ref{theorem-BP_for-subsequence-sums} on $\mathbb{B}_{3}$, we get
	\begin{align*}
		&\left| S_{\alpha}(\mathbb{A}) \right| \\ 
		&\geq \left| 2^{\wedge}(A\cup \{0\}) \right| + \left| S_{\alpha}(\mathbb{B}_{3}) \right|	\\
		&\geq 
		\begin{cases}
			a_{k}+k-1 + (k-1)\left(\sum_{i=1}^{k-2} r_{i}+r_{k-1}-1-\alpha\right) +1 & \text{ if } a_{k} \leq 2k-3 \vspace{0.2cm}\\
			(\theta+1)(k+1)-6 + (k-1)\left(\sum_{i=1}^{k-2} r_{i}+r_{k-1}-1-\alpha\right) +1 & \text{ if } a_{k} \geq 2k-2,
		\end{cases}	\\
		&= 
		\begin{cases}
			a_{k}-k+1 + \sum_{i=1}^{k} ir_{i} - \sum_{i=1}^{k-1} ir_{i} + (k-1)\left(\sum_{i=1}^{k-1} r_{i}-\alpha\right) & \text{ if } a_{k} \leq 2k-3 \vspace{0.2cm}\\
			\theta (k+1)-k-3 + \sum_{i=1}^{k} ir_{i} - \sum_{i=1}^{k-1} ir_{i} + (k-1)\left(\sum_{i=1}^{k-1} r_{i}-\alpha\right) & \text{ if } a_{k} \geq 2k-2.
		\end{cases}	
	\end{align*}
	
	\textbf{Case III} ($r_{k-1} = r_{k} = 1$). In this case, the sequence $\mathbb{B}_{3}$ has $k-2$ distinct elements and the vector $s_{3}$ has $k-2$ terms. Thus, $m \leq k-2$ for both $\mathbb{A}$ and $\mathbb{B}_{3}$. Applying Theorem \ref{Lev restricted thm} on $A\cup \{0\}$ and Theorem \ref{theorem-BP_for-subsequence-sums} on $\mathbb{B}_{3}$, we get
	\begin{align*}
		&\left| S_{\alpha}(\mathbb{A}) \right| \\ 
		&\geq \left| 2^{\wedge}(A\cup \{0\}) \right| + \left| S_{\alpha}(\mathbb{B}_{3}) \right|	\\
		&\geq 
		\begin{cases}
			a_{k}+k-1 + \sum_{i=1}^{k-2} ir_{i} - \sum_{i=1}^{m} ir_{i} + m\left(\sum_{i=1}^{m} r_{i}-\alpha\right) +1 & \text{ if } a_{k} \leq 2k-3 \vspace{0.2cm}\\
			(\theta+1)(k+1)-6 + \sum_{i=1}^{k-2} ir_{i} - \sum_{i=1}^{m} ir_{i} + m\left(\sum_{i=1}^{m} r_{i}-\alpha\right)+1 & \text{ if } a_{k} \geq 2k-2,
		\end{cases}	\\
		&= 
		\begin{cases}
			a_{k}-k+1 + \sum_{i=1}^{k} ir_{i} - \sum_{i=1}^{m} ir_{i} + m\left(\sum_{i=1}^{m} r_{i}-\alpha\right) & \text{ if } a_{k} \leq 2k-3 \vspace{0.2cm}\\
			\theta (k+1)-k-3 + \sum_{i=1}^{k} ir_{i} - \sum_{i=1}^{m} ir_{i} + m\left(\sum_{i=1}^{m} r_{i}-\alpha\right) & \text{ if } a_{k} \geq 2k-2.
		\end{cases}	
	\end{align*}
\end{proof}

\begin{theorem}\label{Freiman-Theorem-for-S_alpha-sequence-4}
	Let $k \geq 3$. Let $\mathbb{A} = \{a_{1},a_{2},\ldots,a_{k}\}_{\vec{r}}$ be a sequence of positive integers such that $a_{1} < a_{2} < \cdots < a_{k}$, $\vec{r} = (r_{1}, r_{2}, \ldots, r_{k})$, $r_{i} \geq 1$ for all $i\in [1,k-2]$, $r_{k-1} = 1$ and $r_{k} \geq 2$. Let $d(A)=1$. Let $ \alpha < \sum_{i=1}^{k} r_{i}-2$ be a positive integer. Then the following holds.
	\begin{enumerate}
		\item[\upshape(1)] If $ \alpha < \sum_{i=1}^{k-2} r_{i}$, then there exists an integer $m \in [1,k-2]$ such that $\sum_{i=1}^{m-1} r_{i} \leq \alpha < \sum_{i=1}^{m} r_{i} $ and 
		\begin{equation*}
			\left|S_{\alpha}(\mathbb{A})\right| 
			\geq \begin{cases}
				&a_{k}+k + \sum_{i=1}^{k-2} ir_{i} - \sum_{i=1}^{m} ir_{i} + m\left(\sum_{i=1}^{m} r_{i}-\alpha\right) + (k-m+1)(r_{k}-1)\\ & \text{ if } a_{k} \leq 2k-3 \vspace{0.2cm} \\
				&\theta(k+1)+k-4 + \sum_{i=1}^{k-2} ir_{i} - \sum_{i=1}^{m} ir_{i} + m\left(\sum_{i=1}^{m} r_{i}-\alpha\right) + (k-m+1)(r_{k}-1)\\ & \text{ if } a_{k} \geq 2k-2.
			\end{cases}
		\end{equation*}
		
		\item[\upshape(2)] If $ \sum_{i=1}^{k-2} r_{i} \leq \alpha < \sum_{i=1}^{k} r_{i} -2$, then 
		\begin{equation*}
			\left|S_{\alpha}(\mathbb{A})\right| \geq 
			\begin{cases}
				a_{k}-k+2 + (k-1)\left(\sum_{i=1}^{k} r_{i}-\alpha\right) & \text{ if } a_{k} \leq 2k-3 \vspace{0.2cm} \\
				\theta(k+1)-k-2 + (k-1)\left(\sum_{i=1}^{k} r_{i}-\alpha\right) & \text{ if } a_{k} \geq 2k-2.
			\end{cases}	
		\end{equation*}
	\end{enumerate}
\end{theorem}

\begin{proof}
	Set $\mathbb{B}_{4} = \{a_{1}, a_{2}, \ldots, a_{k-2}\}_{\vec{s_{4}}}$ with $\vec{s_{4}} = (r_{1}, r_{2}, \ldots, r_{k-2})$. If $ \alpha < \sum_{i=1}^{k-2} r_{i}$, then $2^{\wedge}(A\cup \{0\})$ and $(S^{\alpha}_{1}(\mathbb{B}_{4}) + a_{k-1} + a_{k})$ are two disjoint subsets of $S_{1}^{\alpha}(\mathbb{A})$. For $1\leq i \leq r_{k}-1$ and $1\leq j \leq k-m-1$, define
	\[v_{i,j} = \left(\sum_{\ell=1}^{m} r_{\ell}-\alpha\right) a_{m} + \sum_{t=m+1, t\neq k-j}^{k-1} r_{t}a_{t} + (r_{k-j}-1)a_{k-j} + (i+1)a_{k},\]
	\[v_{i, k-m} = \left(\sum_{\ell=1}^{m} r_{\ell}-\alpha-1\right) a_{m} + \sum_{t=m+1}^{k-1} r_{t}a_{t} + (i+1)a_{k},\] and 
	\[v_{i, k-m+1} = \left(\sum_{\ell=1}^{m} r_{\ell}-\alpha\right) a_{m} + \sum_{t=m+1}^{k-1} r_{t}a_{t} + (i+1)a_{k}.\]
	It is easy to see that \[ v_{i,1} < v_{i,2}<\cdots < v_{i,k-m-1} < v_{i,k-m} < v_{i,k-m+1} < v_{i+1,1}.\] 
	Therefore, the elements $v_{i, j}$ are all distinct, all are in the set $S_{1}^{\alpha}(\mathbb{A})$ and bigger than the elements of $2^{\wedge}(A\cup \{0\})$ and $S^{\alpha}_{1}(\mathbb{B}_{4}) + a_{k-1} + a_{k}$. Therefore, we have
	\begin{align*}
		\left| S_{\alpha}(\mathbb{A}) \right| 
		&= \left| S^{\alpha}_{1}(\mathbb{A}) \right| + 1 \\ 
		&\geq \left| 2^{\wedge}(A\cup \{0\}) \right| + \left| S^{\alpha}_{1}(\mathbb{B}_{4}) \right| + \left|\bigcup_{i=1}^{r_{k}-1}\bigcup_{j=1}^{k-m+1} v_{i,j} \right| + 1 \\ 
		&= \left| 2^{\wedge}(A\cup \{0\}) \right| + \left| S_{\alpha}(\mathbb{B}_{4}) \right| + \sum_{i=1}^{r_{k}-1}\sum_{j=1}^{k-m+1} 1.	
	\end{align*}
	By Theorem \ref{Lev restricted thm} and Theorem \ref{theorem-BP_for-subsequence-sums}, we have
	\begin{align*}
		&\left| S_{\alpha}(\mathbb{A}) \right| \\
		&\geq 
		\begin{cases}
			&a_{k}+k-1 + \sum_{i=1}^{k-2} ir_{i} - \sum_{i=1}^{m} ir_{i} + m\left(\sum_{i=1}^{m} r_{i}-\alpha\right) + 1 + (k-m+1)(r_{k}-1)\\ & \text{ if } a_{k} \leq 2(k+1)-5 \vspace{0.2cm} \\
			&(\theta+1)(k+1)-6 + \sum_{i=1}^{k-2} ir_{i} - \sum_{i=1}^{m} ir_{i} + m\left(\sum_{i=1}^{m} r_{i}-\alpha\right) + 1 + (k-m+1)(r_{k}-1)\\ & \text{ if } a_{k} \geq 2(k+1)-4,
		\end{cases}	\\
		&= 
		\begin{cases}
			&a_{k}+k + \sum_{i=1}^{k-2} ir_{i} - \sum_{i=1}^{m} ir_{i} + m\left(\sum_{i=1}^{m} r_{i}-\alpha\right) + (k-m+1)(r_{k}-1)\\ & \text{ if } a_{k} \leq 2k-3 \vspace{0.2cm} \\
			&\theta(k+1)+k-4 + \sum_{i=1}^{k-2} ir_{i} - \sum_{i=1}^{m} ir_{i} + m\left(\sum_{i=1}^{m} r_{i}-\alpha\right) + (k-m+1)(r_{k}-1)\\ & \text{ if } a_{k} \geq 2k-2.
		\end{cases}
	\end{align*} 
	
	Now, let $\sum_{i=1}^{k-2} r_{i} \leq \alpha < \sum_{i=1}^{k} r_{i} -2 = \sum_{i=1}^{k-2} r_{i}+r_{k}-1$. Set $a_{k-1}^{\prime}=a_{k}$, $r_{k-1}^{\prime}=r_{k}-1$, and $\mathbb{B}_{5} = \{a_{1}, a_{2}, \ldots, a_{k-2}, a_{k-1}^{\prime}\}_{\vec{s_{5}}}$ with $\vec{s_{5}} = (r_{1}, r_{2}, \ldots, r_{k-2}, r_{k-1}^{\prime})$. Then
	\begin{equation*}
		2^{\wedge}(A\cup \{0\}) \cup  (S^{\alpha}_{1}(\mathbb{B}_{5}) + a_{k-1} + a_{k}) \subset S_{1}^{\alpha}(\mathbb{A}),
	\end{equation*}
	where $2^{\wedge}(A\cup \{0\}) \cap (S^{\alpha}_{1}(\mathbb{B}_{5}) + a_{k-1} + a_{k}) = \emptyset$. Thus, 
	\begin{align*}
		\left| S_{\alpha}(\mathbb{A}) \right| = \left| S^{\alpha}_{1}(\mathbb{A}) \right| + 1 
		\geq \left| 2^{\wedge}(A\cup \{0\}) \right| + \left| S^{\alpha}_{1}(\mathbb{B}_{5}) \right| + 1 
		= \left| 2^{\wedge}(A\cup \{0\}) \right| + \left| S_{\alpha}(\mathbb{B}_{5}) \right|.
	\end{align*} 
	As $\sum_{i=1}^{k-2} r_{i} \leq \alpha < \sum_{i=1}^{k-2} r_{i}+r_{k-1}^{\prime}$, by Theorem \ref{theorem-BP_for-subsequence-sums}, we have
	\begin{align*}
		\left| S_{\alpha}(\mathbb{B}_{5}) \right| 
		\geq (k-1)\left(\sum_{i=1}^{k-2} r_{i} + r_{k-1}^{\prime}-\alpha\right) +1 
		= (k-1)\left(\sum_{i=1}^{k-2} r_{i} + r_{k}-1-\alpha\right) +1.
	\end{align*}
	This equation together with Theorem \ref{Lev restricted thm} applying on $A\cup \{0\}$, give
	\begin{align*}
		\left| S_{\alpha}(\mathbb{A}) \right| 
		&= \left| 2^{\wedge}(A\cup \{0\}) \right| + \left| S_{\alpha}(\mathbb{B}_{5}) \right| \\
		&\geq 
		\begin{cases}
			a_{k}+k-1 + (k-1)\left(\sum_{i=1}^{k-2} r_{i} + r_{k}-1-\alpha\right) + 1 & \text{ if } a_{k} \leq 2k-3 \vspace{0.2cm} \\
			(\theta+1)(k+1)-6 + (k-1)\left(\sum_{i=1}^{k-2} r_{i} + r_{k}-1-\alpha\right) + 1 & \text{ if } a_{k} \geq 2k-2,
		\end{cases}	\\
		&= 
		\begin{cases}
			a_{k}-k+2 + (k-1)\left(\sum_{i=1}^{k} r_{i}-\alpha\right) & \text{ if } a_{k} \leq 2k-3 \vspace{0.2cm} \\
			\theta(k+1)-k-2 + (k-1)\left(\sum_{i=1}^{k} r_{i}-\alpha\right) & \text{ if } a_{k} \geq 2k-2.
		\end{cases}	
	\end{align*} 
	This proves the theorem. 
\end{proof}

In the following three theorems, the sequence $\mathbb{A}$ contains nonnegative integers with $0 \in \mathbb{A}$. 

\begin{theorem}\label{the-last-value-of-alpha-zero}
	Let $k \geq 4$. Let $\mathbb{A} = \{a_{0},a_{1}, \ldots, a_{k-1}\}_{\vec{r}}$ be a sequence of nonnegative integers such that $0 = a_{0}< a_{1}< \cdots < a_{k-1}$, $\vec{r} = (r_{0}, r_{1}, \ldots, r_{k-1})$ and $r_{i} \geq 1$ for all $i \in [0,k-1]$. Let $d(A) = 1$ and $r=\min \{r_{1}, r_{2}, \ldots, r_{k-1}\}$. Let $\alpha = \sum_{i=0}^{k-1} r_{i}-2$. If $r=1$, then
	\begin{equation*}
		\left|S_{\alpha} (\mathbb{A}) \right| \geq \begin{cases}
			a_{k-1} + k -1  &  a_{k-1} \leq 2k-5\\
			(\theta + 1)k - 5 & a_{k-1} \geq 2k-4.
		\end{cases}
	\end{equation*}
	If $r \geq 2$, then 
	\begin{equation*}
		\left|S_{\alpha} (\mathbb{A}) \right|  \geq \begin{cases}
			a_{k-1} + k   &  a_{k-1} \leq 2k-3\\
			3k-3 & a_{k-1} \geq 2k-2.
		\end{cases}
	\end{equation*}
\end{theorem}

\begin{proof}
	If $r=1$, then 
	\[S_{\alpha} (\mathbb{A}) = 2^{\wedge}A  \cup \{0\}. \]
	Therefore, by  Theorem \ref{Lev restricted thm}, we get 
	\begin{align*}
		\left| S_{\alpha}(\mathbb{A}) \right| 
		= \left| 2^{\wedge}A \right| + 1 
		\geq \begin{cases}
			a_{k-1} + k -1  &  a_{k-1} \leq 2k-5\\
			(\theta + 1)k - 5 & a_{k-1} \geq 2k-4.
		\end{cases}
	\end{align*}
	If $r \geq 2$, then
	\[S_{\alpha} (\mathbb{A}) = 2A. \]
	Therefore, by  Theorem \ref{Theorem Freiman II}, we get 
	\begin{align*}
		\left| S_{\alpha}(\mathbb{A}) \right| 
		= \left| 2^{\wedge}A \right| + 1 
		\geq \begin{cases}
			a_{k-1} + k   &  a_{k-1} \leq 2k-3\\
			3k-3 & a_{k-1} \geq 2k-2.
		\end{cases}
	\end{align*}
\end{proof}

\begin{theorem}\label{Freiman-Theorem-for-S_alpha-sequence-5}
	Let $k \geq 4$. Let $\mathbb{A} = \{a_{0},a_{1}, \ldots, a_{k-1}\}_{\vec{r}}$ be a sequence of nonnegative integers such that $0 = a_{0}< a_{1}< \cdots < a_{k-1}$, $\vec{r} = (r_{0}, r_{1}, \ldots, r_{k-1})$ and $r_{i} \geq 1$ for all $i \in [0,k-1]$. Let $d(A) = 1$ and $\min \{r_{1}, r_{2}, \ldots, r_{k-1}\} = r \geq 2$. Let $ \alpha < \sum_{i=0}^{k-1} r_{i}-2$ be a positive integer. Then the following holds.  
	\begin{enumerate}         
		\item[\upshape(1)] If $0 < \alpha < \sum_{i=0}^{k-1} r_{i} - r$, then there exists an integer $m \in [1,k]$ such that $\sum_{i=0}^{m-2} r_{i} \leq \alpha < \sum_{i=0}^{m-1} r_{i}$ and
		\begin{align*}
			\left|S_{\alpha}(\mathbb{A}) \right| \geq 
			&|(r-1)A| + \min \{ a_{k-1}, r(k-2)+1\} + \sum_{i=1}^{k-1} (i-1)r_{i-1} - \sum_{i=1}^{m} (i-1)r_{i-1} \\
			&\quad + (m-1)\left(\sum_{i=1}^{m} r_{i-1} - \alpha\right) + (k-1)(r_{k-1}-r).
		\end{align*}
		
		\item[\upshape(2)] If $\sum_{i=0}^{k-1} r_{i} - r \leq \alpha < \sum_{i=0}^{k-1} r_{i} - 2$, then 
		\begin{align*}
			\left|S_{\alpha}(\mathbb{A}) \right| \geq 
			\begin{cases}
				a_{k-1}-k+2 + (k-1)\left(\sum_{i=1}^{k} r_{i-1}-\alpha\right) & \text{ if } a_{k-1} \leq 2k-3 \vspace{0.2cm}\\
				k-1 + (k-1)\left(\sum_{i=1}^{k} r_{i-1}-\alpha\right) & \text{ if } a_{k-1} \geq 2k-2.
			\end{cases}
		\end{align*}
	\end{enumerate}
\end{theorem}

\begin{proof} 
	Set $\mathbb{B}_{6} := \{a_{1}, a_{2}, \ldots, a_{k-1}\}_{\vec{s_{6}}}$ with $\vec{s_{6}} := (r_{1}, r_{2}, \ldots, r_{k-1})$. Then $d(B_{6})=1$. Observe that $S_{1}^{\alpha}(\mathbb{A}) = S(\mathbb{B}_{6}) \cup \{0\}$ if $0 < \alpha \leq r_{0}$ and $S_{1}^{\alpha}(\mathbb{A}) =S_{1}^{\alpha-r_{0}} (\mathbb{B}_{6}) \cup \{0\}$ if $r_{0} < \alpha \leq \sum_{i=0}^{k-1} r_{i} - 2$. Note also that, if $\alpha > r_{0}$ and $\sum_{i=0}^{m-2} r_{i} \leq \alpha < \sum_{i=0}^{m-1} r_{i}$ for some $m \in [2,k]$, then  $\sum_{i=1}^{m-2} r_{i} \leq \alpha-r_{0} < \sum_{i=1}^{m-1} r_{i} $. Therefore the integer $m$ for $\mathbb{A}$ will work as $m-1$ for $\mathbb{B}_{6}$.
	
	If $0 < \alpha \leq r_{0}$, we have from Theorem \ref{Freiman-theorem-2-for-vector-r} that  
	\begin{align*}
		\left| S_{\alpha}(\mathbb{A}) \right|
		&=\left| S_{1}^{\alpha}(\mathbb{A}) \right| \\ 
		&= \left| S(\mathbb{B}_{6}) \right| + 1 \\
		&\geq |(r-1)A| + \min \{ a_{k-1}, r(k-2)+1\} -1 + \sum_{i=1}^{k-2} ir_{i} + (k-1)(r_{k-1}-r) +1 \\
		&= |(r-1)A| + \min \{ a_{k-1}, r(k-2)+1\} + \sum_{i=1}^{k-1} (i-1)r_{i-1} +  (k-1)(r_{k-1}-r).
	\end{align*}
	
	If $r_{0} < \alpha < \sum_{i=0}^{k-1} r_{i} - 2$, then we have two possibility, either $r_{0} < \alpha < \sum_{i=0}^{k-1} r_{i} - r $ or $\sum_{i=0}^{k-1} r_{i} - r \leq \alpha < \sum_{i=0}^{k-1} r_{i} - 2$. In the first case, we have $0 < \alpha - r_{0} < \sum_{i=1}^{k-1} r_{i} - r$, so by  Theorem \ref{Freiman-Theorem-for-S_alpha-sequence-1}, we obtain  
	\begin{align*}
		&\left| S_{\alpha}(\mathbb{A}) \right|\\
		&=\left| S_{1}^{\alpha}(\mathbb{A}) \right| \\ 
		&= \left| S_{1}^{\alpha-r_{0}}(\mathbb{B}_{6}) \right| + 1 \\
		&= \left| S_{\alpha-r_{0}}(\mathbb{B}_{6}) \right| \\
		&\geq |(r-1)A| + \min \{ a_{k-1}, r(k-2)+1\} + \sum_{i=1}^{k-2} ir_{i} - \sum_{i=1}^{m-1} ir_{i} + (m-1)\left(\sum_{i=1}^{m-1} r_{i} - (\alpha-r_{0})\right) \\
		&\quad+ (k-1)(r_{k-1}-r) \\
		&= |(r-1)A| + \min \{ a_{k-1}, r(k-2)+1\} + \sum_{i=1}^{k-1} (i-1)r_{i-1} - \sum_{i=1}^{m} (i-1)r_{i-1} \\
		&\quad + (m-1)\left(\sum_{i=1}^{m} r_{i-1} - \alpha\right) + (k-1)(r_{k-1}-r).
	\end{align*}
	In the second case, we have $\sum_{i=1}^{k-1} r_{i} - r \leq \alpha -r_{0} < \sum_{i=1}^{k-1} r_{i} - 2$. By Theorem \ref{Freiman-Theorem-for-S_alpha-sequence-2}, we get \begin{align*}
		\left| S_{\alpha}(\mathbb{A}) \right|
		&= \left| S_{\alpha-r_{0}}(\mathbb{B}_{6}) \right| \\
		&\geq 
		\begin{cases}
			a_{k-1}-k+2 + (k-1)\left(\sum_{i=1}^{k-1} r_{i}-(\alpha-r_{0})\right) & \text{ if } a_{k-1} \leq 2k-3 \vspace{0.2cm}\\
			k-1 + (k-1)\left(\sum_{i=1}^{k-1} r_{i}-(\alpha-r_{0})\right) & \text{ if } a_{k-1} \geq 2k-2,
		\end{cases}\\
		&= 
		\begin{cases}
			a_{k-1}-k+2 + (k-1)\left(\sum_{i=1}^{k} r_{i-1}-\alpha\right) & \text{ if } a_{k-1} \leq 2k-3 \vspace{0.2cm}\\
			k-1 + (k-1)\left(\sum_{i=1}^{k} r_{i-1}-\alpha\right) & \text{ if } a_{k-1} \geq 2k-2.
		\end{cases}
	\end{align*}
	This completes the proof of the theorem. 
\end{proof}

The following theorem is for $r=1$. 

\begin{theorem}\label{Freiman-Theorem-for-S_alpha-sequence-6}
	Let $k \geq 4$. Let $\mathbb{A} = \{a_{0},a_{1}, \ldots, a_{k-1}\}_{\vec{r}}$ be a sequence of nonnegative integers such that $0 = a_{0}< a_{1}< \cdots < a_{k-1}$, $\vec{r} = (r_{0}, r_{1}, \ldots, r_{k-1})$ and $r_{i} \geq 1$ for all $i \in [0,k-1]$. Let $d(A) = 1$ and $\min \{r_{1}, r_{2}, \ldots, r_{k-1}\} = r = 1$. Let $ \alpha < \sum_{i=0}^{k-1} r_{i}-2$ be a positive integer. Then the following holds.  
	
	\begin{enumerate}
		\item [\upshape(1)] If $0 < \alpha < \sum_{i=0}^{k-2} r_{i} - 1$ with $r_{k-1} \neq 1$ and $r_{k} \neq 1$ or $r_{k-1} = r_{k} = 1$ or $r_{k-1} \neq 1$ and $r_{k} = 1$, then there exists an integer $m \in [1,k]$ such that $\sum_{i=0}^{m-2} r_{i} \leq \alpha < \sum_{i=0}^{m-1} r_{i} $ and   
		\begin{align*}
			&\left| S(\mathbb{A}) \right| \\
			&\geq     
			\begin{cases}
				a_{k-1}-k+2 + \sum_{i=1}^{k} (i-1)r_{i-1} - \sum_{i=1}^{m} (i-1)r_{i-1} + (m-1)\left(\sum_{i=1}^{m} r_{i-1}-\alpha\right) & \text{ if } a_{k-1} \leq 2k-5 \vspace{0.2cm}\\
				\theta k-k-2 + \sum_{i=1}^{k} (i-1)r_{i-1} - \sum_{i=1}^{m} (i-1)r_{i-1} + (m-1)\left(\sum_{i=1}^{m} r_{i-1}-\alpha\right) & \text{ if }  a_{k-1} \geq 2k-4.
			\end{cases}
		\end{align*}
		
		\item [\upshape(2)] If $\sum_{i=0}^{k-2} r_{i} - 1 \leq \alpha  < \sum_{i=0}^{k-1} r_{i} - 2$ with $r_{k-1} \neq 1$ and $r_{k} \neq 1$ or $r_{k-1} = r_{k} = 1$ or $r_{k-1} \neq 1$ and $r_{k} = 1$, then  
		\begin{equation*}
			\left| S(\mathbb{A}) \right| \geq    
			\begin{cases}
				a_{k-1}-k+1 + (k-1)\left(\sum_{i=1}^{k} r_{i-1}-\alpha\right) & \text{ if } a_{k-1} \leq 2k-5 \vspace{0.2cm}\\
				\theta k-k-3+  (k-1)\left(\sum_{i=1}^{k} r_{i-1}-\alpha\right) & \text{ if }  a_{k-1} \geq 2k-4.
			\end{cases}
		\end{equation*}
		
		\item[\upshape(3)] If $0 < \alpha \leq r_{0}$ with $r_{k-2} = 1$ and $r_{k-1} \neq 1$, then 
		\begin{align*}
			&\left|S_{\alpha}(\mathbb{A})\right| \geq    
			\begin{cases}
				a_{k-1}-k+2 + \sum_{i=1}^{k} (i-1)r_{i-1} & \text{ if } a_{k-1} \leq 2k-5 \vspace{0.2cm}\\
				\theta k-k-2 + \sum_{i=1}^{k} (i-1)r_{i-1} & \text{ if }  a_{k-1} \geq 2k-4.
			\end{cases}
		\end{align*}
		
		\item[\upshape(4)] If $r_{0} < \alpha < \sum_{i=0}^{k-3} r_{i}$ with $r_{k-2} = 1$ and $r_{k-1} \neq 1$, then there exists an integer $m \in [1,k]$ such that $\sum_{i=0}^{m-2} r_{i} \leq \alpha < \sum_{i=0}^{m-1} r_{i} $ and 
		\begin{align*}
			&\left|S_{\alpha}(\mathbb{A})\right| \\
			&\geq    
			\begin{cases}
				&a_{k-1}+1 + \sum_{i=1}^{k-1} (i-1)r_{i-1} - \sum_{i=1}^{m} (i-1)r_{i-1} + (m-1)\left(\sum_{i=1}^{m} r_{i-1}-\alpha\right) + (k-m)(r_{k-1}-1)\\ & \text{ if } a_{k-1} \leq 2k-5 \vspace{0.2cm} \\
				&\theta k-3 + \sum_{i=1}^{k-1} (i-1)r_{i-1} - \sum_{i=1}^{m} (i-1)r_{i-1} + (m-1)\left(\sum_{i=1}^{m} r_{i-1}-\alpha\right) + (k-m)(r_{k-1}-1)\\ & \text{ if } a_{k} \geq 2k-4.
			\end{cases}
		\end{align*}
		
		\item[\upshape(5)] If $\sum_{i=0}^{k-3} r_{i} \leq \alpha < \sum_{i=0}^{k-1} r_{i} - 2$ with $r=r_{k-2} = 1$ and $r_{k-1} \neq 1$, then 
		\begin{align*}
			\left|S_{\alpha}(\mathbb{A})\right| 
			&\geq    
			\begin{cases}
				a_{k-1}-k+3 + (k-2)\left(\sum_{i=1}^{k} r_{i-1}-\alpha\right) & \text{ if } a_{k-1} \leq 2k-5 \vspace{0.2cm} \\
				\theta k-k-1 + (k-2)\left(\sum_{i=1}^{k} r_{i-1}-\alpha\right) & \text{ if } a_{k-1} \geq 2k-4.
			\end{cases}
		\end{align*}
	\end{enumerate}
\end{theorem}

\begin{proof} 
	Set $\mathbb{B}_{7} := \{a_{1}, a_{2}, \ldots, a_{k-1}\}_{\vec{s_{7}}}$ with $\vec{s_{7}} := (r_{1}, r_{2}, \ldots, r_{k-1})$. Then $d(B_{7})=1$. Observe that $S_{1}^{\alpha}(\mathbb{A}) = S(\mathbb{B}_{7}) \cup \{0\}$ if $0 < \alpha \leq r_{0}$ and $S_{1}^{\alpha}(\mathbb{A}) =S_{1}^{\alpha-r_{0}} (\mathbb{B}_{7}) \cup \{0\}$ if $r_{0} < \alpha \leq \sum_{i=0}^{k-1} r_{i} - 2$. Note also that, if $\alpha > r_{0}$ and $\sum_{i=0}^{m-2} r_{i} \leq \alpha < \sum_{i=0}^{m-1} r_{i}$ for some $m \in [2,k]$, then  $\sum_{i=1}^{m-2} r_{i} \leq \alpha-r_{0} < \sum_{i=1}^{m-1} r_{i} $. Therefore the integer $m$ for $\mathbb{A}$ will work as $m-1$ for $\mathbb{B}_{7}$.

	\textbf{Case I} ($r_{k-1} \neq 1$ and $r_{k} \neq 1$ or $r_{k-1} = r_{k} = 1$ or $r_{k-1} \neq 1$ and $r_{k} = 1$). If $0 < \alpha \leq r_{0}$, we have from Theorem \ref{Freiman-theorem-1-for-vector-r} that 
	\begin{align*}
		\left| S_{\alpha}(\mathbb{A}) \right|
		&= \left| S_{1}^{\alpha}(\mathbb{A}) \right| \\
		&= \left| S(\mathbb{B}_{7}) \right| + 1 \\
		& \geq    
		\begin{cases}
			a_{k-1}-(k-1) + \sum_{i=1}^{k-1} ir_{i} + 1 & \text{ if } a_{k-1} \leq 2(k-1)-3 \vspace{0.2cm}\\
			\theta k -(k-1)-4 + \sum_{i=1}^{k-1} ir_{i} + 1 & \text{ if }  a_{k-1} \geq 2(k-1)-2.
		\end{cases} \\
		&= \begin{cases}
			a_{k-1}-k+2 + \sum_{i=1}^{k} (i-1)r_{i-1}  & \text{ if } a_{k-1} \leq 2k-5 \vspace{0.2cm}\\
			\theta k-k-2 + \sum_{i=1}^{k} (i-1)r_{i-1} & \text{ if }  a_{k-1} \geq 2k-4.
		\end{cases}
	\end{align*}
	
	If $r_{0} < \alpha < \sum_{i=0}^{k-1} r_{i} - 2$, then we have two possibilitiy, either $0 < \alpha -r_{0} < \sum_{i=1}^{k-2} r_{i} - 1$ or $\sum_{i=1}^{k-2} r_{i} - 1 \leq \alpha - r_{0} < \sum_{i=1}^{k-1} r_{i} - 2$. If $\alpha -r_{0} < \sum_{i=1}^{k-2} r_{i} - 1$, then  by  Theorem \ref{Freiman-Theorem-for-S_alpha-sequence-3}, we get  
	\begin{align*}
		&\left| S_{\alpha}(\mathbb{A}) \right| \\
		&= \left| S_{1}^{\alpha}(\mathbb{A}) \right| \\
		&= \left| S_{1}^{\alpha-r_{0}}(\mathbb{B}_{7}) \right| +1 \\
		&= \left| S_{\alpha-r_{0}}(\mathbb{B}_{7}) \right| \\
		& \geq    
		\begin{cases}
			a_{k-1}-k+2 + \sum_{i=1}^{k-1} ir_{i} - \sum_{i=1}^{m-1} ir_{i} + (m-1)\left(\sum_{i=1}^{m-1} r_{i}-(\alpha-r_{0})\right) & \text{ if } a_{k-1} \leq 2(k-1)-3 \vspace{0.2cm}\\
			\theta k-k-2 + \sum_{i=1}^{k-1} ir_{i} - \sum_{i=1}^{m-1} ir_{i} + (m-1)\left(\sum_{i=1}^{m-1} r_{i}-(\alpha-r_{0})\right) & \text{ if }  a_{k-1} \geq 2(k-1)-2,
		\end{cases} \\
		& =    
		\begin{cases}
			a_{k-1}-k+2 + \sum_{i=1}^{k} (i-1)r_{i-1} - \sum_{i=1}^{m} (i-1)r_{i-1} + (m-1)\left(\sum_{i=1}^{m} r_{i-1}-\alpha\right) & \text{ if } a_{k-1} \leq 2k-5 \vspace{0.2cm}\\
			\theta k-k-2 + \sum_{i=1}^{k} (i-1)r_{i-1} - \sum_{i=1}^{m} (i-1)r_{i-1} + (m-1)\left(\sum_{i=1}^{m} r_{i-1}-\alpha\right) & \text{ if }  a_{k-1} \geq 2k-4.
		\end{cases}
	\end{align*}
	
	If $\sum_{i=1}^{k-2} r_{i} - 1 \leq \alpha - r_{0} < \sum_{i=1}^{k-1} r_{i} - 2$, then again by Theorem \ref{Freiman-Theorem-for-S_alpha-sequence-3}, we get  
	\begin{align*}
		&\left| S_{\alpha}(\mathbb{A}) \right| \\
		&= \left| S_{\alpha-r_{0}}(\mathbb{B}_{7}) \right| \\
		& \geq    
		\begin{cases}
			a_{k-1}-k+1 + (k-1)\left(\sum_{i=1}^{k-1} r_{i}-(\alpha-r_{0})\right) & \text{ if } a_{k-1} \leq 2(k-1)-3 \vspace{0.2cm}\\
			\theta k-k-3 +  (k-1)\left(\sum_{i=1}^{k-1} r_{i}-(\alpha-r_{0})\right) & \text{ if }  a_{k-1} \geq 2(k-1)-2,
		\end{cases} \\
		& =    
		\begin{cases}
			a_{k-1}-k+1 + (k-1)\left(\sum_{i=1}^{k} r_{i-1}-\alpha\right) & \text{ if } a_{k-1} \leq 2k-5 \vspace{0.2cm}\\
			\theta k-k-3+  (k-1)\left(\sum_{i=1}^{k} r_{i-1}-\alpha\right) & \text{ if }  a_{k-1} \geq 2k-4.
		\end{cases}
	\end{align*}
	
	\textbf{Case II} $(r_{k-2} = 1$ and $r_{k-1} \neq 1)$. If $0 < \alpha \leq r_{0}$, we have from Theorem \ref{Freiman-theorem-1-for-vector-r} that 
	\begin{align*}
		\left| S_{\alpha}(\mathbb{A}) \right|
		&= \left| S(\mathbb{B}_{7}) \right| + 1 \\
		&\geq \begin{cases}
			a_{k-1}-(k-1) + \sum_{i=1}^{k-1} ir_{i} +1 & \text{ if } a_{k-1} \leq 2(k-1)-3 \vspace{0.2cm}\\
			\theta k-(k-1)-4 + \sum_{i=1}^{k-1} ir_{i} +1 & \text{ if }  a_{k-1} \geq 2(k-1)-2.
		\end{cases}\\
		&= \begin{cases}
			a_{k-1}-k+2 + \sum_{i=1}^{k} (i-1)r_{i-1} & \text{ if } a_{k-1} \leq 2k-5 \vspace{0.2cm}\\
			\theta k-k-2 + \sum_{i=1}^{k} (i-1)r_{i-1} & \text{ if }  a_{k-1} \geq 2k-4.
		\end{cases}
	\end{align*}
	
	If $r_{0} < \alpha < \sum_{i=0}^{k-3} r_{i}$, then $0 < \alpha-r_{0} < \sum_{i=1}^{k-3} r_{i}$. Thus, it follows from the proof of Theorem \ref{Freiman-Theorem-for-S_alpha-sequence-4} that 
	\begin{align*}
		&\left| S(\mathbb{A}) \right| \\
		&= \left| S_{\alpha-r_{0}}(\mathbb{B}_{7}) \right| \\
		& \geq    
		\begin{cases}
			&a_{k-1}+k-1 + \sum_{i=1}^{k-3} ir_{i} - \sum_{i=1}^{m-1} ir_{i} + (m-1)\left(\sum_{i=1}^{m-1} r_{i}-(\alpha-r_{0})\right) + (k-m)(r_{k-1}-1)\\ & \text{ if } a_{k-1} \leq 2(k-1)-3 \vspace{0.2cm} \\
			&\theta k+k-5 + \sum_{i=1}^{k-3} ir_{i} - \sum_{i=1}^{m-1} ir_{i} + (m-1)\left(\sum_{i=1}^{m-1} r_{i}-(\alpha-r_{0})\right) + (k-m)(r_{k-1}-1)\\ & \text{ if } a_{k} \geq 2(k-1)-2,
		\end{cases} \\
		& =    
		\begin{cases}
			&a_{k-1}+1 + \sum_{i=1}^{k-1} (i-1)r_{i-1} - \sum_{i=1}^{m} (i-1)r_{i-1} + (m-1)\left(\sum_{i=1}^{m} r_{i-1}-\alpha\right) + (k-m)(r_{k-1}-1)\\ & \text{ if } a_{k-1} \leq 2k-5 \vspace{0.2cm} \\
			&\theta k-3 + \sum_{i=1}^{k-1} (i-1)r_{i-1} - \sum_{i=1}^{m} (i-1)r_{i-1} + (m-1)\left(\sum_{i=1}^{m} r_{i-1}-\alpha\right) + (k-m)(r_{k-1}-1)\\ & \text{ if } a_{k} \geq 2k-4.
		\end{cases}
	\end{align*}
	
	If $\sum_{i=0}^{k-3} r_{i} \leq \alpha < \sum_{i=0}^{k-1} r_{i}-2$, then we have again from Theorem \ref{Freiman-Theorem-for-S_alpha-sequence-4} that 
	\begin{align*}
		\left| S(\mathbb{A}) \right| 
		&= \left| S_{\alpha-r_{0}}(\mathbb{B}_{7}) \right| \\
		& \geq    
		\begin{cases}
			a_{k-1}-k+3 + (k-2)\left(\sum_{i=1}^{k-1} r_{i}-(\alpha-r_{0})\right) & \text{ if } a_{k-1} \leq 2(k-1)-3 \vspace{0.2cm} \\
			\theta k-k-1 + (k-2)\left(\sum_{i=1}^{k-1} r_{i}-(\alpha-r_{0})\right) & \text{ if } a_{k-1} \geq 2(k-1)-2,
		\end{cases}	\\
		& =    
		\begin{cases}
			a_{k-1}-k+3 + (k-2)\left(\sum_{i=1}^{k} r_{i-1}-\alpha\right) & \text{ if } a_{k-1} \leq 2k-5 \vspace{0.2cm} \\
			\theta k-k-1 + (k-2)\left(\sum_{i=1}^{k} r_{i-1}-\alpha\right) & \text{ if } a_{k-1} \geq 2k-4.
		\end{cases}
	\end{align*}
	This completes the proof of the theorem. 
\end{proof}

\section*{Acknowledgment}
The first author would like to thank the Council of Scientific and Industrial Research (CSIR), India for providing the grant to carry out the research with Grant No. 09/143(0925)/2018-EMR-I.

\bibliographystyle{amsplain}

\end{document}